\documentclass[12pt]{amsart} 

\makeatletter
\newif\ifarxiv
\@ifclassloaded{journal}{\arxivfalse}{\arxivtrue}
\makeatother

\usepackage[utf8]{inputenc}

\usepackage[margin=3cm]{geometry}

\usepackage{amsmath, amssymb, mathtools, color,verbatim,bbm}
\ifarxiv \usepackage{amsthm} \fi
\usepackage{units}
\usepackage{array}
\usepackage{blkarray}
\usepackage{graphicx}
\usepackage{cite}
\usepackage{mathrsfs}
\usepackage{appendix}
\usepackage{makecell}
\usepackage{arydshln}
\usepackage{nicefrac}
\usepackage{cancel}
\usepackage{soul}
\usepackage{float}
\usepackage{enumitem}
\ifarxiv
\else
\setlist[enumerate]{leftmargin=.5in}
\setlist[itemize]{leftmargin=.5in}
\fi
\usepackage[math]{cellspace}
\usepackage{booktabs}
\usepackage{makecell}
\usepackage{pdflscape}

\usepackage{tikz}
\usetikzlibrary{cd}
\usetikzlibrary{backgrounds}
\usetikzlibrary{calc}
\tikzset{curve/.style={settings={#1},to path={(\tikztostart)
    .. controls ($(\tikztostart)!\pv{pos}!(\tikztotarget)!\pv{height}!270:(\tikztotarget)$)
    and ($(\tikztostart)!1-\pv{pos}!(\tikztotarget)!\pv{height}!270:(\tikztotarget)$)
    .. (\tikztotarget)\tikztonodes}},
    settings/.code={\tikzset{quiver/.cd,#1}
        \def\pv##1{\pgfkeysvalueof{/tikz/quiver/##1}}},
    quiver/.cd,pos/.initial=0.35,height/.initial=0}

\usepackage{hyperref}
\hypersetup{
    colorlinks,
    linkcolor={red!50!black},
    citecolor={blue!50!black},
    urlcolor={blue!80!black}
}  

\ifarxiv
\newtheorem{theorem}{Theorem}
\numberwithin{theorem}{section}
\newtheorem{proposition}[theorem]{Proposition}
\newtheorem{lemma}[theorem]{Lemma}
\newtheorem{corollary}[theorem]{Corollary}

\theoremstyle{definition}

\newtheorem{definition}[theorem]{Definition}
\newtheorem{example}[theorem]{Example}
\theoremstyle{remark}
\newtheorem{remark}[theorem]{Remark}
\else
\newcommand{\newjournalremark}[2]{
  \theoremstyle{plain}
  \theoremheaderfont{\normalfont\itshape}
  \theorembodyfont{\normalfont}
  \theoremseparator{.}
  \theoremsymbol{}
  \newtheorem{#1}[theorem]{#2}
}
\newjournalremark{example}{Example}
\newjournalremark{remark}{Remark}
\fi

\newcommand{\RR}{\mathbb{R}}
\newcommand{\CC}{\mathbb{C}}
\newcommand{\PP}{\mathbb{P}}

\newcommand{\Mcal}{\mathcal{M}}

\newcommand{\tr}{\mathrm{tr}}

\newcommand{\Sym}{\mathrm{Sym}_2}
\newcommand{\MLD}{\mathrm{MLD}_{F}(X)}
\newcommand{\myMLD}{\operatorname{MLD}}

\newcommand{\bu}{{\boldsymbol u}}
\newcommand{\bv}{{\boldsymbol v}}

\newcommand{\vectorspace}{\mathcal{L}}
\newcommand{\secondvectorspace}{\mathcal{W}}
\newcommand{\MLE}{\operatorname{MLE}}
\newcommand{\dashto}{\dashrightarrow}
\newcommand{\sm}{\mathrm{sm}}
\renewcommand{\Im}{\operatorname{Im}}
\newcommand{\id}{\operatorname{id}}

\newcommand{\Span}{\operatorname{Span}}
\newcommand{\isom}{\cong}
\newcommand{\Hom}{\operatorname{Hom}}
\newcommand{\rad}{\operatorname{rad}}
\newcommand{\SL}{\operatorname{SL}}
\newcommand{\Lcal}{\mathcal L}
\newcommand{\pa}{\operatorname{pa}}

\newcommand{\diag}{\operatorname{diag}}

\newcommand{\resp}[1]{#1}

\definecolor{forest}{RGB}{11,128,35}

\ifarxiv
\title[Diff.\ equations for Gaussian statistical models with rational MLE]{Differential equations for Gaussian statistical models with rational maximum likelihood estimator}
\author[Am\'endola, Gustafsson, Kohn, Marigliano, Seigal]{Carlos Am\'{e}ndola, Lukas Gustafsson, Kathl\'{e}n  Kohn,\\ Orlando Marigliano, Anna Seigal}
\date{\today}

\newenvironment{ourkeywords}{\footnotesize\textbf{Keywords:}}{}
\newenvironment{MSCcodes}{\footnotesize\textbf{MSC codes:}}{}

\else

\newenvironment{ourkeywords}{\begin{keywords}}{\end{keywords}}
\newcommand{\qedhere}{}

\title{Differential equations for Gaussian statistical models \\ with rational maximum likelihood estimator}
\author{Carlos Am\'{e}ndola\footnote{Technische Universität Berlin, Germany (\email{amendola@math.tu-berlin.de})}, Lukas Gustafsson\footnote{KTH Royal Institute of Technology, Sweden (\email{lukasgu@kth.se}, \email{kathlen@kth.se}, \email{orlandom@kth.se})}, Kathl\'{e}n  Kohn\footnotemark[2],\\ Orlando Marigliano\footnotemark[2], Anna Seigal\footnote{Harvard University, USA (\email{aseigal@seas.harvard.edu})}}

\headers{Diff.\ equations for Gaussian statistical models with rational MLE}{Am\'endola, Gustafsson, Kohn, Marigliano, Seigal}
\date{\today}
\fi

\begin{document}

\begin{abstract} 
We study multivariate Gaussian statistical models whose maximum likelihood estimator (MLE) is a rational function of the observed data.
We establish a one-to-one correspondence between such models and the solutions to a nonlinear first-order partial differential equation (PDE).
Using our correspondence, we reinterpret familiar classes of models with rational MLE, 
such as directed (and decomposable undirected) Gaussian graphical models. We also find new models with rational MLE.
For linear concentration models with rational MLE, we show that homaloidal polynomials from birational geometry lead to solutions to the PDE.
We thus shed light on the problem of classifying Gaussian models with rational MLE by relating it to the open problem in birational geometry of classifying homaloidal polynomials.
\end{abstract}

\maketitle

\begin{ourkeywords}
maximum likelihood degree, multivariate Gaussian, homaloidal polynomial
\end{ourkeywords}


\section{Introduction}

The \emph{maximum likelihood degree} (ML degree) of a statistical model
is the number of complex critical points of the likelihood function given general data. It
was introduced in \cite{CHKS06} for discrete statistical models and in \cite{SU10} for Gaussian models.
For models with ML degree one, the likelihood function for general data has a unique critical point, which is the maximum likelihood estimate given that data. The \emph{maximum likelihood estimator} (MLE) is the function that maps data to its maximum likelihood estimate. A model has ML degree one if and only if the MLE is a rational function of the data. 

A classification of discrete statistical models of ML degree one was obtained in \cite{Huh14(1),DMS21(1)}. 
The classification follows a two step procedure.
The first step shows that discrete  models of ML degree one are the solutions to a system of partial differential equations (PDEs) \cite[Lemma 15]{Huh14(1)}.
The second step uses Horn uniformization to parametrize all solutions to the PDEs \cite[Lemma 16]{Huh14(1)}.

This article pursues the analogous classification for Gaussian models of ML degree one.
We show that varieties of Gaussian ML degree one are in bijection with solutions to a nonlinear first-order PDE.
This is analogous to the first step of \cite{Huh14(1)}.
For the second step, we show that parametrizing the solutions to our PDE specializes to an open problem in classical algebraic geometry concerning \emph{homaloidal polynomials}, which have been extensively studied \cite{Dol00, DP03, Huh14(2), SST21}.
We leave it as an open problem for future work to parametrize all solutions to our \emph{homaloidal PDE}. 

{\bf Other related work.}
The concept of ML degree fits into the perspective of \emph{likelihood geometry}, introduced by Huh and Sturmfels \cite{huh2014likelihood} and part of the field of algebraic statistics \cite{sullivant2018algebraic}.
In the discrete setting, Huh \cite{Huh13} studied the ML degree of a very affine variety and showed that, for smooth varieties, the ML degree is equal to the topological Euler characteristic. ML degrees have been studied for toric varieties, which in statistics correspond to discrete exponential families~\cite{ABBGHHNRID19}. \resp{In particular, Garcia-Puente and Sottile showed that the discrete statistical model associated to a toric patch has ML degree one if and only if a certain Laurent polynomial associated to the toric patch is homaloidal with respect to the logarithmic toric differential~\cite{garcia-puente-sottile}. This highlights another setting where a homaloidal-like property is related to models with ML degree one. Graf von Bothmer, Ranestad, and  Sottile then used this relationship to classify two-dimensional toric patches whose discrete statistical models have ML degree one~\cite{bothmer-ranestad-sottile}.}

In the Gaussian setting, the focus has largely been on linear concentration models; i.e., models where the inverse covariance matrices lie on a linear space of symmetric matrices \cite{AGKMS21,amendola2021likelihood,boege2021reciprocal,dye2021maximum,davies2021coloured,eur2021reciprocal,jiang2021linear,MMMSV20,MMW21}. Further work relates ML degrees to Euler characteristics in the Gaussian setting \cite{DRGS22}. The ML degree was also studied for \emph{exponential varieties} \cite{MSUZ16}.

\medskip

We now introduce our main concepts. Let $\mathcal M$ be a semi-algebraic subset of the cone of real positive definite symmetric $m\times m$ matrices. We regard $\mathcal M$ as a set of concentration (inverse covariance) matrices parametrizing a mean-centered Gaussian statistical model. We typically assume that $\Mcal$ is \emph{scaling-invariant}; that is, $K \in \mathcal M$ implies 
$\lambda K \in \mathcal M$ for all positive scalars $\lambda \in \RR$. 
That is, rescaling of measurements does not affect membership in the model. Given data $Y_1,\dots,Y_n \in \RR^m$, the Gaussian \emph{log-likelihood function} is
\begin{equation}
    \label{eqn:lS}
\ell_S : \Mcal  \to \RR,\quad  K \mapsto \frac{n}{2} (\log \det (K) - \tr (K S) - m\log(2\pi)).
\end{equation} 
where $S=\frac{1}{n} \sum_{i=1}^n Y_i Y_i^\top$ is the sample \emph{covariance} matrix.
In contrast,
the points in the model are parametrized by \emph{concentration} matrices. 
We choose this parametrization to make the trace term linear in both arguments:
the trace term is a canonical pairing of a vector space with its dual. We will find this perspective helpful later.

We pass from statistical models to algebraic varieties.
We denote the Zariski closure of the complexification of $\mathcal M$ by $X$. 
The log-likelihood function extends to a function $\ell_S$ on $X$. 
The ML degree of $X$ is the number of complex critical points of this function for a general complex symmetric $m \times m$ matrix $S \in \Sym(\mathbb C^m)$.

The rest of this paper is organized as follows. We reinterpret the Gaussian log-likelihood function in a coordinate-free context, using projective varieties and homogeneous polynomials in Section~\ref{sec:likelihood}. We introduce the homaloidal PDE in Section~\ref{sec:homaloidalPDE}.
Our main result is a one-to-one correspondence between varieties of Gaussian ML degree one and solutions to the PDE (Theorem \ref{thm:pde}). Equally important is a parallel result on another PDE which is uniquely satisfied by the MLEs of such varieties (Theorem \ref{thm:3conditionspsi}). 
We study solutions to the homaloidal PDE coming from linear spaces in Section \ref{sec:linear_and_homaloidal} and find that the linear spaces of ML degree one correspond to homaloidal polynomials.  
We revisit known families of Gaussian graphical models with ML degree one in Section~\ref{sec: graphicalmodels}, where we also formulate our main result in terms of usual Gaussian statistical models (Corollary~\ref{cor:pdeSym}). We take first steps towards parametrizing the solutions to the PDE in Section~\ref{sec:solutions}, and produce new ML degree one varieties in Section~\ref{sec:constructing}. Code for our computational examples can be found at the MathRepo repository \url{https://mathrepo.mis.mpg.de/GaussianMLDeg1}.

\section{A coordinate-free Gaussian likelihood function}
\label{sec:likelihood}

This section gives a coordinate-free definition of the Gaussian ML degree of an embedded variety, extending ideas from \cite{MSUZ16,DRGS22}.
In this paper, $\mathcal L$ denotes a finite-dimensional $\mathbb C$-vector space.
Before focusing on the ML degree, we recall that the Jacobian of a rational map $\varphi : X\dashto Y$ between algebraic varieties $X$ and~$Y$ is a rational map
\[
J\varphi : X\dashto \Hom(TX, TY),
\quad p\mapsto J_p\varphi : T_pX\to T_pY.
\]
If $X = \mathcal L$, we identify the tangent space $T_p \mathcal L$ with~$\mathcal L$. In this case, the Jacobian of $\varphi:\mathcal L\dashto \mathbb C$ is the gradient
\[
\nabla\varphi : \mathcal L\dashto \mathcal L^\ast,\quad
p\mapsto \nabla_p\varphi.
\]

We give a coordinate-free definition of the Gaussian ML degree of a variety \resp{$X\!\subseteq \Sym(\CC^m)$}.
This definition replaces the space of symmetric matrices by an arbitrary ambient linear space $\mathcal{L}$.
This offers a simplification when $X$ lies in a low-dimensional affine subspace of the space of symmetric matrices. \resp{Indeed, it allows us to consider only data belonging to that linear subspace instead of having to account for the whole of $\Sym$.}
\resp{We show in Proposition \ref{prop:statistical} that the coordinate-free Gaussian ML degree is equivalent to the Gaussian ML degree previously studied in the literature. The latter is defined as the number of critical points $K$ of the log-likelihood function \eqref{eqn:lS}, see e.g.~\cite{SU10} and~\cite[Def.\ 2.1.4, Prop.\ 2.1.12]{lectures-algebraic-statistics}.}

In the definition below, we count the critical points of a function $\ell$ on $X$. Note that~$\ell$ can only be differentiated along $X$ at its smooth locus $X^{\sm}$.

\begin{definition}\label{ml-degree-general}
Let $\vectorspace$ be a
finite-dimensional
$\CC$-vector space, let $X\subseteq \vectorspace$ be an affine variety, and fix a polynomial $F$ on $\vectorspace$. For $\bu\in\vectorspace^\ast$ and $p\in \vectorspace\setminus V(F)$,  the  \emph{(Gaussian) log-likelihood} is
\[
\ell_{F,\bu}(p) \coloneqq \log(F(p)) - \bu(p).
\]
The  \emph{(Gaussian) ML degree} of $X$ with respect to $F$, denoted $\MLD$, is the number of critical points of $\ell_{F,\bu}$ over the domain ${X^{\sm}\setminus V(F)}$ for general $\bu$.
\end{definition}

\begin{remark}
To define the log-likelihood $\ell_{F,\bu}$ as a function on $\vectorspace\setminus V(F)$ we must choose a local branch of the logarithm around each $p\in \vectorspace\setminus V(F)$. The gradient $\nabla_p \ell_{F,\bu}\in \vectorspace^{\ast}$ does not depend on this choice. The number $\MLD$ is the cardinality of the set
\[
\mathcal C\coloneqq \{p\in X^{\mathrm{sm}}\setminus V(F) \mid
T_p X \subseteq \ker(\nabla_p \ell_{F,\bu})\}
\]
for general $\bu$. This cardinality does not depend on (general) $\bu$.
Moreover,
$\mathcal C$ is reduced as a scheme, cf.~\cite[Lemma~2.5]{DRGS22}.
\end{remark}

\resp{
\begin{remark}\label{irreducible-mld}
The ML degree of a variety with respect to a polynomial $F$ is the sum of the ML degrees of its connected components since $T_p X\subseteq \ker(\nabla_p \ell_{F,\bu})$ is a local condition in $p$. The above still holds if we replace ``connected'' by ``irreducible'',
since the intersection of irreducible components is never part of the smooth locus. 
For this reason, we assume from now on that $X$ is irreducible.
\end{remark}
}

\resp{The following lemma gives the sense in which $\MLD$ is independent of the choice of embedding of $X$ into a linear space.}

\resp{\begin{lemma}\label{linear-transformation}
Let $\mathcal L$ and $\mathcal L'$ be finite-dimensional $\mathbb C$-vector spaces, $X\subseteq \mathcal L$ an affine variety, $G$ a polynomial on $\mathcal L'$, and $\mathcal A:\mathcal L\to\mathcal L'$ an affine-linear embedding. Define $F = G\circ \mathcal A$. Then $\myMLD_{F}(X) = \myMLD_{G}(\mathcal A(X))$.
\end{lemma}
\begin{proof}
Let $\mathcal A^*:(\mathcal L')^*\to \mathcal L^*$ be the dual affine-linear map. For $\bu'\in(\mathcal L')^*$ and $p\in X$,
\[
\ell_{G,\bu'}(\mathcal A(p))
= \log G (\mathcal A(p)) - \mathcal \bu'(\mathcal A(p))
= \log F (p) - \mathcal A^*(\bu')(p)
= \ell_{F,\mathcal A^*(\bu')}(p).
\]
Thus, the critical points of $\ell_{G,\bu'}$ over $\mathcal A(X)$ correspond one-to-one with the critical points of $\ell_{F,\mathcal A^*(\bu')}$ over $X$. Since $\mathcal A^*$ is surjective, a general $\bu\in\mathcal L^*$ has the form $\mathcal A^*(\bu')$ for a general $\bu'\in (\mathcal L')^*$. This completes the proof.
\end{proof}
}

We now prove the equivalence of the usual and coordinate-free 
ML degrees. 

\begin{proposition}\label{prop:statistical} 
Let $\det$ denote the determinant on $\Sym(\mathbb C^m)$.
\begin{enumerate}[label=(\alph*)]
\item
Let $X$ be a subvariety of $\Sym(\mathbb C^m)$.
Then $\mathrm{MLD}_{\det}(X)$ is the usual Gaussian ML degree of $X$.
\item Conversely, let $\vectorspace$ be a finite-dimensional
$\mathbb C$-vector space, $X\subseteq \vectorspace$ a variety, and $F$ a polynomial on $\vectorspace$. Then there exists an affine-linear embedding $\mathcal{A}: \vectorspace \to \Sym(\mathbb C^m)$ such that $\MLD = \mathrm{MLD}_{\det}(\mathcal{A}(X))$.
\end{enumerate}
\end{proposition}

\begin{proof} 
Set $F=\det$ in 
Definition~\ref{ml-degree-general}.
A generic linear form $\bu$ has the form
$\bu(K)=\tr(KS)$, for $S\in \Sym(\mathbb C^m)$ generic, since the trace induces an isomorphism of $\Sym(\mathbb C^m)$ with its dual. This proves (a).

For (b), let $k=\dim(\mathcal L)$. There exist symmetric matrices $A_0,\dotsc,A_k$, such that
\begin{align} \label{eq:detRepresentation}
    F(x) = \det A(x), \quad
\text{where} \quad
A(x)\coloneqq A_0 + \sum_{i=1}^k x_i A_i,
\end{align}
by \cite[Theorem~13]{AW21}.
Let $m$ be the size of the $A_i$.
If the $A_1,\dotsc, A_k$ are linearly independent, then sending a basis of $\mathcal L$ to the tuple $(A_0+A_1,\dotsc,A_0+A_k)$ gives an affine-linear embedding $\mathcal{A}:\mathcal L \to \Sym(\CC^m)$ with $\myMLD_F(X) = \myMLD_{\det}(\mathcal A(X))$, \resp{by Lemma~\ref{linear-transformation}}. Otherwise, let $\{A_r,\dotsc, A_k\}$ be a maximal linearly independent subset of $\{A_1,\dotsc, A_k\}$, possibly after reordering.  If $r\leq m$, let
\[
B \coloneqq \begin{pmatrix}
D & A(x)\\
A(x) & 0
\end{pmatrix}
\quad \text{where }
D \coloneqq \diag(x_1,\dotsc,x_{r-1},0,\dotsc,0) \in \Sym(\CC^m).
\]
Then
$\det B = \det A(x)^2 = F^2$.
Furthermore, we have $B(x) = B_0 + \sum_{i=1}^k x_i B_i$ with linearly independent $B_1,\dotsc, B_k$. Define the embedding $\mathcal A:\mathcal L \to \Sym(\mathbb C^{2m})$ by sending a basis of $\mathcal L$ to $(B_0+B_1,\dotsc, B_0+B_k)$. Observe that 
\begin{align}\label{eq:change-of-f}
    \myMLD_{F}(X) = \myMLD_{\lambda F^n}(X)
\end{align}
for any nonzero scalar $\lambda$ and positive integer $n$, since $\nabla_p \ell_{\lambda F^n, \bu}
=
n\nabla_p\log F - \bu$. Hence $\myMLD_{\det}(\mathcal A(X)) = \myMLD_{F^2}(X) = \myMLD_{F}(X)$ and we are done.

If $r>m$, we increase the size of $A(x)$ by taking the direct sum with a suitably large identity matrix. We then apply the same argument to the resulting matrix $B$. 
\end{proof}

\resp{The previous proof shows that the choice of $\mathcal A$ does not matter for computing $\MLD$, provided $\mathcal A$ satisfies 
\[
\lambda F^n = \det \circ \mathcal A \quad \text{for some } \lambda \in \CC^* \text{ and } n \in \mathbb{N}.\]}
We give a hands-on example of the construction in the previous proof when the matrices $A_i$ are linearly dependent.

\begin{example}
Let $F(x_1,x_2,x_3) = (x_1 + x_3)(x_2 + x_3)$. Then $F(x) = \det A(x)$, where 
$$ A(x) = x_1 A_1 + x_2 A_2 + x_3 A_3 = \begin{pmatrix} 
x_1 + x_3 & 0 \\ 0 & x_2 + x_3 
\end{pmatrix}.
$$
However, the matrices $A_i$ relate via $A_1 = A_3 - A_2$. Using the proof of Proposition~\ref{prop:statistical}(b) with $r=m=2$, we write $F^2$ as the determinant of the matrix 
\[
\sum_{i=1}^3 x_i B_i = 
\begin{pmatrix}
x_1 & 0 & x_1+x_3 & 0\\
0& 0& 0& x_2+x_3 \\
x_1+x_3 & 0& 0& 0\\
0& x_2+x_3 &0 &0
\end{pmatrix}.
\]
Here, the $B_i$ are linearly independent.
\end{example}

To apply Proposition~\ref{prop:statistical} to a Gaussian statistical model $\mathcal M$ in the cone of positive-definite real symmetric $m\times m$ matrices, take the Zariski closure $X$ of $\mathcal M$ in $\Sym(\mathbb C^m)$ and let $\mathcal A$ be its affine span.
The restriction $F \coloneqq \det|_{\mathcal{A}}$ of the determinant to $\mathcal{A}$ satisfies 
\begin{align*}
    \MLD = \mathrm{MLD}_{\det}(X),
\end{align*}
which is the usual Gaussian ML degree of $\mathcal M$. The log-likelihood $\ell_S$ for a covariance matrix $S$ becomes the function $\ell_{\det, \bu}$ from Definition~\ref{ml-degree-general} with $\bu = \tr(S-)$.

\begin{example}
Let $\mathcal M\subseteq \Sym(\mathbb R^3)$ be the set of diagonal 
positive definite real symmetric $3\times 3$ 
matrices and let $X \isom \CC^3$ be its complex Zariski closure. 
The determinant restricts to the linear space $\mathcal{A}=X$ to give the cubic $F(x_1, x_2, x_3) = x_1 x_2 x_3$.
Then
\[
\mathrm{MLD}_{\det}(X) = \mathrm{MLD}_{x_1x_2x_3}(\CC^3) = 1, 
\]
where we verify the latter equality as follows.
For a generic linear form $\bu = u_1x_1 + u_2x_2 + u_3x_3$, 
the log-likelihood
$\ell_{F,\bu}(x) = \log(x_1x_2x_3)-\bu(x)$
has a unique critical point, namely the unique solution $(x_1,x_2,x_3) = (\nicefrac{1}{u_1}, \nicefrac{1}{u_2}, \nicefrac{1}{u_3}) \in \mathbb C^3$ to the equation
\[
\nabla\ell_{F,\bu}(x) =
\begin{pmatrix}
    \nicefrac{1}{x_1} - u_1 \\
    \nicefrac{1}{x_2} - u_2 \\
    \nicefrac{1}{x_3} - u_3
\end{pmatrix}
= 0. 
\] 
\end{example} 

\begin{remark}
Let $\mathcal L$ be a finite-dimensional \emph{real} vector space, $X\subseteq \mathcal L$ a variety, and $F$ a polynomial on $\mathcal L$. Then there still exists an affine-linear embedding $\mathcal{A}:\mathcal L\to \Sym(\mathbb R^m)$ such that $F$ becomes the restriction of the determinant. However, it is not always possible to find an embedding $\mathcal{A}$ that intersects the positive definite cone. Thus $\mathcal{A}$ need not yield a statistical model. When such an $\mathcal{A}$ exists, the polynomial $F$ is said to possess a \emph{definite symmetric determinantal representation}.
The problem of computing a definite symmetric determinantal representation of a real polynomial is well-studied in convex algebraic geometry, due to its connections to semidefinite programming and the generalized Lax conjecture \cite{blekherman2012semidefinite,CD20,D20}.
\end{remark}

From now on, we consider \emph{projective} varieties $X$, reflecting the assumption that our statistical models are scaling-invariant. In this case, the affine span of the affine cone $C_X$ over $X$ is a linear space and restricting the determinant to that linear subspace yields a homogeneous polynomial. We thus assume that $F$ is homogeneous and set \[
\myMLD_{F}(X) \coloneqq \myMLD_{F}(C_X).
\]

We now make a connection between the coordinate-free ML degree of a variety $X$ \resp{with respect to a certain polynomial $F$} and the Euclidean distance degree $\mathrm{EDD}(X)$, which counts the critical points of the Euclidean distance function to $X$ from a general point~\cite{euclidean-distance}.
This builds on the perspective in \cite[Corollary~3.3]{MRW21} and \cite[Theorem~1.1]{DRGS22}.

\begin{proposition}
\label{prop:EDdegree}
    Let $Q = \sum_{i=0}^n x_i^2$ be the Fermat quadric on $\mathbb{P}^n$ and $X \subseteq \mathbb{P}^n$ be a projective variety.
    Then the Euclidean distance degree of $X$ and the ML degree of $X$ with respect to $Q$ coincide; i.e.,   
$$ 
\mathrm{EDD}(X) = \mathrm{MLD}_{Q}(X).
$$
\end{proposition}

\begin{proof}
We identify $\CC^{n+1}$ and $(\CC^{n+1})^\ast$ via $Q$ and thus regard $x$ and $\bu$ as points of $\CC^{n+1}$.
We find a bijection between the critical points of $\ell_{Q,\bu}$ and the critical points of the Euclidean distance function $d_\bu^2(x)\coloneqq Q(x - \bu)$.
That is, between the sets
\[
\mathcal C(d^2_{\bu}) \coloneqq \{x\resp{ \in C_X} \mid x-\bu \perp T_x C_X\}
\quad\text{and}\quad
\mathcal C(\ell_\bu) \coloneqq \{x \resp{ \in C_X} \mid \frac{2x}{Q(x)}-\bu \perp T_x C_X\}.
\]
\resp{Here, $T_x C_X$ refers to the embedded tangent space, so that $x$ is an element of $T_x C_X$. Applying $\langle -,x \rangle$ to the definition of $\mathcal C(d_\bu^2)$ we see that $Q(x)-\langle \bu, x\rangle = 0$ for $x\in \mathcal C(d_\bu^2)$. Doing the same for $\mathcal C(\ell_\bu)$, we deduce $2 - \langle \bu, x\rangle = 0$ for $x\in \mathcal C(\ell_\bu)$.}
The required bijection is given in both directions by the rational involution $\mu : \CC^{n+1}\dashto\CC^{n+1}$ defined by $\mu(x)=\nabla\log Q=2x/Q(x)$.
\end{proof}

The goal of this paper is to characterize irreducible projective  varieties with ML degree one. 
Let $X\subseteq \PP(\vectorspace)$ be a projective variety and $F$ a homogeneous polynomial on $\vectorspace$. 
We denote by $C_X$ the affine cone over $X$.
If $\MLD = 1$, there is a map
\begin{equation}
\label{eqn:define_mle_map} 
    \MLE_{X,\vectorspace,F} : \vectorspace^\ast \dashrightarrow C_X^{\sm}\setminus V(F)
\end{equation}
that takes a general $\bu\in\vectorspace^{\ast}$ to the unique critical point of $\ell_{F,\bu}$ along $C_X$.
If $X$ is the Zariski closure of a statistical model in $\mathcal L = \Sym(\mathbb C^m)$ and $F=\det$, the map $\MLE_{X,\mathcal L, F}$ is the \emph{maximum likelihood estimator}.

\begin{proposition}\label{prop:mle-rational}
\resp{Let $X\subseteq \PP(\vectorspace)$ be a projective variety and $F$ a homogeneous polynomial on $\vectorspace$ such that $\MLD = 1$. Then} the map $\MLE_{X,\vectorspace,F}$ is rational and dominant.
Furthermore, if $X\subseteq \PP(\secondvectorspace)$ for some subspace $\secondvectorspace\subseteq \vectorspace$ and $\pi:\vectorspace^\ast\to\secondvectorspace^\ast$ is the restriction to $\secondvectorspace$, then
\[
{\MLE_{X,\vectorspace,F}} = {\MLE_{X,\secondvectorspace,F|_\secondvectorspace}} \circ \pi.
\]
\end{proposition}

\begin{proof}
Consider the incidence variety $\mathfrak X \subseteq (C_X^{\sm}\setminus V(F))\times \vectorspace^\ast$ of all pairs $(p,\bu)$ such that $p$ is a critical point of $\ell_{F,\bu}$. Let $\pi_i$ be the projections from $\mathfrak X$ onto its $i$th factor.
By assumption, $\pi_2$ is birational. For a general $p\in C^\sm_X\setminus V(F)$ and $\bu\coloneqq \nabla_p \log F$, we have $\nabla_p \ell_{F,\bu} = \nabla_p\log F - \bu = 0$.
Thus, $p$ is a critical point of $\ell_{F,\bu}$, which shows that  $\pi_1$ is dominant.
Hence, the map 
$\MLE_{X,\mathcal{L}, F} = \pi_1\circ\pi_2^{-1}$ is rational and dominant.
The second statement holds since $\ell_{F,\bu} = \ell_{F|_\secondvectorspace, \bu|_\secondvectorspace}$ on $C_X^{\sm}\setminus V(F)$.
\end{proof}

\begin{remark}
As seen in the proof of Proposition~\ref{prop:mle-rational}, the map $p\mapsto \nabla_p \log F$ sends points in $C_X^\sm\setminus V(F)$ to linear forms $\bu$ such that $(p,\bu)$ is a critical pair. This is analogous to the fact that a statistical model can be viewed as a set of empirical probability distributions such that every model point is its own maximum likelihood estimate.
\end{remark}

A projective variety of ML degree one has rational maximum likelihood estimator, 
by Proposition~\ref{prop:mle-rational}.
We show the converse.

\begin{proposition} \label{rational-mle-mld-one}
Let $X\subseteq \PP(\mathcal L)$ be a projective variety and
\[
\Psi:\Lcal^\ast\dashto C_X^{\sm}\setminus V(F)
\]
a dominant rational map such that, for general $\bu\in\Lcal^*$, the point $\Psi(\bu)$ is a critical point of $\ell_{F,\bu}$ along $C_X$. Then $\myMLD_F(X) = 1$ and $\Psi = \MLE_{X,\mathcal L, F}$.
\end{proposition}
\begin{proof}
Let $\mathfrak X$ be the incidence variety from the proof of Proposition~\ref{prop:mle-rational}. Suppose $(p,\bu)$ and $(p,\bv)$ lie in $\mathfrak X$ for some $\bu, \bv\in \vectorspace^{\ast}$. Then 
 the linear forms $\bu$ and $\bv$ are the same when restricted to $T_p C_X$. 
Thus $\pi_1:\mathfrak X\to C^{\sm}_X\setminus V(F)$ makes $\mathfrak X$ into a vector bundle whose fiber at a base point $p$ is isomorphic to the co-normal space $N^\ast_pC_X$. Hence $\dim(\mathfrak X) = \dim(\vectorspace^\ast)$. The variety $X$ is irreducible since $\Psi$ is dominant, hence so is $\mathfrak X$. Define $\sigma:\vectorspace^\ast\dashto\mathfrak X$ by $\sigma(\bu) = (\Psi(\bu), \bu)$. Then $\sigma$ is a right-inverse of $\pi_2: \mathfrak X\to \vectorspace^{\ast}$, thus an injective rational map between irreducible varieties of the same dimension, thus birational with inverse~$\pi_2$. The map $\pi_2$ has degree one, thus $\MLD = 1$ and $\Psi = \pi_1\circ\pi_2^{-1} = \MLE_{X,\mathcal L, F}$.
\end{proof}

 \section{The homaloidal PDE}
\label{sec:homaloidalPDE}

In this section, we state and prove our main results. We characterize the map $\MLE_{X,\mathcal L, F}$ as the solution to a  PDE in Theorem~\ref{thm:3conditionspsi}. We then introduce the \emph{homaloidal PDE} and show that its solutions are precisely projective varieties of Gaussian ML degree one in Theorem~\ref{thm:pde}.
Our justification for the adjective \emph{homaloidal} is explained after Corollary~\ref{cor:homaloidal_solutions}. 

\begin{theorem} 
\label{thm:3conditionspsi}
Let $X \subseteq \PP (\vectorspace)$ be an irreducible projective variety and $F$ a homogeneous polynomial on $\vectorspace$. Then $\MLD = 1$ if and only if there exists a dominant rational map $\Psi: \vectorspace^\ast \dashrightarrow C_X$
that satisfies
\begin{enumerate}
\item[(a)] $\Psi(t \bu) = t^{-1} \Psi(\bu)$ for all $t\in \CC \setminus \{ 0\}$,
\item[(b)] $\nabla_{\bu} \log (F \circ \Psi) = - \Psi(\bu)$
for general $\bu\in \vectorspace^\ast$.
\end{enumerate}
In this case,
$\Psi$ is the function $\mathrm{MLE}_{X,\vectorspace,F}$.
\end{theorem}

Before the proof, recall that Euler's homogeneous function theorem says that a homogeneous function $f$ on $\CC^n$ of degree $d$ satisfies
$\sum_{i=1}^n \frac{\partial f(x)}{\partial x_i}\cdot x_i = d\cdot f(x)$ for all $x\in \CC^n$.
Translated to a coordinate-free formulation for a vector space and its dual, the result says that if
$\varphi:\mathcal L\dashto \mathbb C$ is a homogeneous rational function of degree $d$ with $p\in \mathcal L$ then
\[
(\nabla_p\varphi)(p) = d\cdot \varphi(p). \]
Similarly, given $\psi:\mathcal L^\ast \dashto \mathbb C$ and $\bu\in \mathcal L^{\ast}$, the result says that $\bu(\nabla_{\bu}\psi) = d\cdot \psi(\bu)$.

\begin{proof}[Proof of Theorem~\ref{thm:3conditionspsi}]
Let $d\coloneqq \deg(F)$. We first show that if $\MLD = 1$ then
the map
$\Psi\coloneqq\MLE_{X, \vectorspace, F}$
satisfies (a) and (b).
This map is rational and dominant, 
by Proposition~\ref{prop:mle-rational}.
For (a), let $t\in \mathbb C\setminus \{ 0\}$, and $p\coloneqq\Psi(\bu)$.
Then
\[
\nabla_{t^{-1}p}\ell_{F,t\bu}
= \frac{\nabla_{t^{-1}p}\,F}{F(t^{-1}p)} - t\bu
= \frac{(t^{-1})^{d-1}\,\nabla_p F}{(t^{-1})^d F(p)} - t\bu
= t\,\nabla_p\ell_{F,\bu}.
\]
Since $C_X$ is a cone, $T_pC_X$ and $T_{t^{-1}p}C_X$ are equal as linear subspaces of $\vectorspace$. Thus
\[
T_{t^{-1}p}C_X =
T_p C_X \subseteq \ker(\nabla_p \ell_{F,\bu}) = \ker(\nabla_{t^{-1}p}\ell_{F,t\bu}).
\]
Since $\MLD = 1$, we must have $t^{-1}p =\Psi(t\bu)$. 
Thus, $\Psi$ is homogeneous of degree $-1$.
For (b), we have 
\[\nabla_{\bu}\log (F \circ \Psi)
= (\nabla_{\Psi(\bu)}\log F) \circ J_{\bu} \Psi
= \bu \circ J_{\bu} \Psi
\]
because $\Psi(\bu)$ is a critical point of $\ell_{F,\bu}$ and $\Im(J_u \Psi)\subseteq T_{\Psi(u)}C_X$.
Furthermore, since $C_X$ is a cone we have $\Psi(\bu)\in T_{\Psi(\bu)}C_X$, which yields
\[
0 = (\nabla_{\Psi(\bu)} \ell_{F,\bu}) (\Psi(\bu))
= \frac{(\nabla_{\Psi(\bu)}F)(\Psi(\bu))}{F(\Psi(\bu))} - \bu(\Psi(\bu))
= d - \bu(\Psi(\bu)),
\]
where the last equality follows from Euler's homogeneous function theorem.
Consider the canonical pairing $\beta : \vectorspace^\ast\times \vectorspace\to \mathbb C$. Then
$\nabla_{(\bu, p)}\beta(\dot\bu,\dot p) = \dot\bu(p) + \bu(\dot p)$ for all $\bu,\dot \bu\in \mathcal L^{\ast}$ and all $p,\dot p\in \mathcal L$ and so $\nabla_{(\bu, p)}\beta = \beta(-,p) + \beta(\bu,-)$. Define the function $\varphi\coloneqq \bu(\Psi(\bu)) :\vectorspace^\ast\to\mathbb C$. Then $\varphi = \beta \circ (\id, \Psi)$. By the above computation, $\varphi$ is a constant function equal to $d$. Thus
\begin{align*}
0 &= \nabla_{\bu}\varphi = (\nabla_{(\bu,\Psi(\bu))}\beta) \circ (\id, J_{\bu}\Psi)\\
&= (\beta(-,\Psi(\bu)) + \beta(\bu,-))
\circ(\id, J_{\bu}\Psi) = \Psi(\bu) + \bu\circ J_{\bu}\Psi.
\end{align*}
Hence $\nabla_{\bu}(\log F\circ \Psi) = \bu \circ J_{\bu}\Psi = -\Psi(\bu)$.

For the converse statement,
let $\Psi:\vectorspace^\ast\dashto C_X$ be a dominant rational map satisfying (a) and (b), and take a general $\bu\in\vectorspace^\ast$.
The map $F\circ\Psi$ is homogeneous of degree $-d$.
By Euler's homogeneous function theorem, $\bu(\nabla_{\bu}(F\circ \Psi)) = -d\cdot F(\Psi(\bu))$. By (b),
\[
\bu(\Psi(\bu)) = \bu(-\nabla_{\bu}(\log F \circ \Psi)) = - \frac{\bu(\nabla_{\bu}(F\circ \Psi))}{F(\Psi(\bu))} = d.
\]
Thus $0 = \nabla_{\bu}\varphi = \Psi(\bu) + \bu\circ J_{\bu}\Psi$ as before, so $(\nabla_{\Psi(\bu)}\log F - \bu)\circ J_{\bu}\Psi = 0$ by (b) again. Since $\Psi$ is dominant and $\bu$ is general, $J_{\bu}\Psi$ maps surjectively onto $T_{\Psi(\bu)}C_X$. Hence we have $T_{\Psi(\bu)}C_X\subseteq \ker\nabla_{\Psi(\bu)}\ell_{\bu,F}$, so $\Psi(\bu)$ is a critical point of $\ell_{\bu, F}$. In other words, $\Psi$ satisfies the hypotheses of Proposition~\ref{rational-mle-mld-one}. Thus, $\myMLD_F(X) = 1$ and $\Psi = \MLE_{X,\mathcal L, F}$.
\end{proof} 

The next proposition shows that $\MLE_{X,\mathcal L, F}$ can be recovered from its projectivization.

\begin{proposition}
Let $X\subseteq \PP(\vectorspace)$ be a projective variety and $F$ a homogeneous polynomial on $\vectorspace$ such that $\MLD = 1$. Fix $\bu\in\mathcal L^\ast$ and let $p\in C_X$ be any representative of the image of $[\bu]$ under the map
$
\PP(\MLE_{X,\mathcal L, F}): \PP(\mathcal L^\ast)\dashto X.
$
Then
\[
\MLE_{X,\mathcal L, F}(\bu) = \frac{\deg(F)}{\bu(p)}\, p.
\]
\end{proposition}
\begin{proof}
Let $q\coloneqq \MLE_{X,\mathcal L, F}(\bu)$. Then $\bu(q) = \deg F$, as shown in the proof of Theorem~\ref{thm:3conditionspsi}. Since $q = \lambda p$ for some $\lambda\in\mathbb C$, we have $q/\bu(q) = p/\bu(p)$, thus
\[
q = \frac{\deg(F)}{\bu(q)} q = \frac{\deg(F)}{\bu(p)}p.\qedhere
\]
\end{proof}

The next definition will be used throughout the rest of the paper.

\begin{definition}
Let $\mathcal L$ be a
finite-dimensional
$\mathbb C$-vector space and $F$ a homogeneous polynomial on $\mathcal L$. The \emph{homaloidal PDE} is the nonlinear first-order PDE:
\begin{equation}\label{pde}
\Phi = F\circ (-\nabla \log \Phi),
\quad
\Phi:\vectorspace^\ast\dashto\mathbb C\textup{ rational and homogeneous.}
\end{equation}
\end{definition}
\begin{remark}
\hfill
\begin{enumerate}
\item[(a)] Every solution $\Phi$ to~\eqref{pde} is homogeneous of degree $-\deg(F)$, since $\nabla\log\Phi$ is homogeneous of degree $-1$.
\item[(b)] If $\lambda\in\mathbb C\setminus \{ 0\}$ and $\Phi$ is a solution for $(\mathcal L, F)$, then $\lambda\Phi$ is a solution for $(\mathcal L, \lambda F)$.
\end{enumerate}
\end{remark}
We now establish the promised one-to-one correspondence between solutions to the homaloidal PDE and projective varieties of ML degree one. We consider two rational functions to be equal if they agree on a dense open set.

\begin{theorem}\label{thm:pde}
Let $\mathcal L$ be a finite-dimensional  $\CC$-vector space and $F$ a homogeneous polynomial on $\mathcal L$. There is a bijection between the projective varieties $X\subseteq \PP(\mathcal L)$ with $\MLD = 1$ and the solutions $\Phi$ to the homaloidal PDE.
The bijection sends a variety $X$ to the function $$\Phi\coloneqq F\circ \MLE_{X,\mathcal L, F},$$ and a function $\Phi$ to the variety $$X\coloneqq \PP(\overline{\Im(-\nabla\log\Phi)}).$$
A variety $X$ with $\MLD = 1$ and its corresponding solution $\Phi$ are related by
\begin{equation}\label{homaloidal-and-mle}
\MLE_{X,\mathcal L, F} = -\nabla\log\Phi.
\end{equation}
\end{theorem}
\begin{proof}
If $\Phi$ is a solution of the homaloidal PDE~\eqref{pde} then $\Psi\coloneqq -\nabla\log\Phi:\vectorspace^\ast\dashto C_X$ is  rational, homogeneous of degree $-1$, and satisfies
\[
\nabla\log (F\circ \Psi) = \nabla\log(F\circ(-\nabla \log \Phi))
=\nabla \log \Phi = -\Psi,
\]
\resp{the second equality being the homaloidal PDE.}
Let $X\coloneqq \PP(\overline{\Im(\Psi)})$. We have $\MLD = 1$ and $\MLE_{X,\mathcal L,F} = \Psi$, 
by Theorem~\ref{thm:3conditionspsi}. In particular,~\eqref{homaloidal-and-mle} holds.

Conversely, let $X\subseteq \PP(\vectorspace)$ satisfy $\MLD = 1$ and $\Psi\coloneqq \MLE_{X,\vectorspace,F}$. By Theorem~\ref{thm:3conditionspsi}(a), $\Psi$ is homogenous of degree $-1$. Define $\Phi$ to be the restriction of $F\circ \Psi:\vectorspace^{\ast}\dashto \CC$ to the Zariski-dense open set
\[
\{\bu\in\vectorspace^\ast \mid \nabla_{\Psi(\bu)} F \text{ and } J_{\bu} \Psi \text{ have full rank}\}.
\]
Then $\Phi$ is a smooth rational homogeneous function. By Theorem~\ref{thm:3conditionspsi}(b),
\[
-\nabla\log\Phi = -\nabla\log (F\circ \Psi) = \Psi,
\]
thus~\eqref{homaloidal-and-mle} holds.
Applying $F$ to both sides, we get $F(-\nabla\log\Phi) = F \circ \Psi = \Phi$, hence $\Phi$ is a solution of~\eqref{pde}.

To show that the above constructions are mutually inverse, start with~\eqref{homaloidal-and-mle}. Apply $F$ to both sides for one direction and $\PP(\overline{\Im(-)})$ for the other direction.
\end{proof}

One upshot of Theorem~\ref{thm:pde} is that the maximum likelihood estimator $\MLE_{X,\mathcal L, F}$ can be expressed purely in terms of the scalar-valued function $\Phi_{X,\mathcal L, F}$, the solution to the homaloidal PDE associated to $(X,\mathcal L, F)$.
Since the map $\MLE_{X,\mathcal L, F}$ behaves well when restricting to linear subspaces, by Proposition~\ref{prop:mle-rational}, so does the function $\Phi_{X,\mathcal L,F} = F \circ \MLE_{X,\mathcal L, F}$. More precisely,
if $X\subseteq \PP(\secondvectorspace)$ for some subspace $\secondvectorspace\subseteq \vectorspace$ and $\pi:\vectorspace^\ast\to\secondvectorspace^\ast$ is the restriction to $\secondvectorspace$ then 
\[
{\Phi_{X,\vectorspace,F}} = {\Phi_{X,\secondvectorspace,F|_\secondvectorspace}} \circ \pi.
\]

\section{Linear spaces and homaloidal polynomials}
\label{sec:linear_and_homaloidal}

This section studies the case $X=\PP(\mathcal L)$. We seek the homogeneous polynomials $F$ on $\vectorspace$ such that
$
\myMLD_F(\PP(\vectorspace)) = 1.
$
This offers a change of perspective from the usual study of the ML degree: rather than, for fixed $F$, finding the linear spaces $\Lcal$ of ML degree one, we instead fix $\Lcal$ and find the polynomials $F$ with respect to which $\mathcal L$ has ML degree one. As it turns out, this class of polynomials is already present in the literature.

\begin{definition}
\label{def:homaloidal}
Let $\vectorspace$ be a
finite-dimensional
$\mathbb C$-vector space. A homogeneous polynomial $F$ on $\vectorspace$ is \emph{homaloidal} if the rational map $\nabla\log F : \vectorspace\dashto\vectorspace^\ast$ is birational. 
\end{definition}

The above definition follows Dolgachev~\cite[Section 2]{Dol00}. For more recent results on homaloidal polynomials we refer the reader to~\cite{Huh14(2)} and~\cite{SST21}. Our interest in these polynomials comes from the following fact. 

\begin{proposition}\label{lem:homaloidal-equiv}
Let $F$ be a homogeneous polynomial on $\vectorspace$. Then:
\begin{enumerate}
\item[(a)] $\myMLD_F(\PP(\vectorspace)) = \deg(\nabla \log F)$
\item[(b)] $\myMLD_F(\PP(\vectorspace)) = 1$ if and only if $F$ is homaloidal. In this case, $\MLE_{\PP(\vectorspace), \vectorspace, F}$ is the rational inverse of $\nabla \log F$.
\end{enumerate}
\end{proposition}
\begin{proof}
The degree of \resp{the rational map} $\nabla\log F$ is the number of solutions $p\in\vectorspace$ to
\[
\frac{\nabla_p F}{F(p)} - \bu = 0
\]
for generic $\bu\in\vectorspace^\ast$, which is exactly $\myMLD_{F}(\PP(\vectorspace)).$ Thus, $\myMLD_F(\PP(\vectorspace)) = 1$ if and only if $\nabla \log F$ is birational. In this case, its inverse $\Psi$ satisfies
\[
\frac{\nabla_{\Psi(\bu)} F}{F(\Psi(\bu))} - \bu = 0,
\]
for general $\bu\in \vectorspace^\ast$, thus it equals $\MLE_{\PP(\vectorspace),\vectorspace, F}$.
\end{proof}

Part (a) of Proposition \ref{lem:homaloidal-equiv}
is used as a definition of the ML degree in \cite[Definition 2.3.1]{DMS21(2)}. Part (b) allows us to establish a close connection between homaloidal polynomials and solutions to the homaloidal PDE.

\begin{corollary}
\label{cor:homaloidal_solutions}
Let $F$ be homaloidal and $\Psi = (\nabla\log F)^{-1}$. Then $F\circ \Psi$ satisfies the homaloidal PDE.
Conversely, if $\Phi$ is a solution to the homaloidal PDE and if
$X \coloneqq \PP(\overline{\Im(-\nabla\log \Phi)})$ is equal to $\PP(\mathcal L)$,
then $F$ is homaloidal. 
\end{corollary}
\begin{proof}
Both statements follow from Proposition~\ref{lem:homaloidal-equiv} and Theorem~\ref{thm:pde}. For the converse statement, the birational inverse of $\nabla\log F$ is given by $-\nabla \log \Phi$.
\end{proof}

Our motivation for calling~\eqref{pde} the homaloidal PDE comes from the fact that if $\Phi$ is a solution to~\eqref{pde} for an arbitrary subvariety $X\subseteq \PP(\mathcal L)$, then $\Psi \coloneqq -\nabla\log\Phi$ satisfies $\Psi\circ (\nabla\log F) = \id$ on $C_X$, giving a birational inverse of 
$\nabla\log F$ on $C_X$. This extends Definition~\ref{def:homaloidal} to instances beyond the linear space $\vectorspace$. 

In light of Proposition~\ref{lem:homaloidal-equiv}(b), we can restate two results from the literature on homaloidal polynomials as results on linear models of ML degree one, as follows.

\begin{theorem}\label{plane-curves} 
{\rm\cite[Theorem 4]{Dol00}.}
Let $\mathcal L \isom \mathbb C^3$ and let $F$ be a homogeneous polynomial on $\mathcal L$ without repeated factors. Then $\myMLD_F(\PP(\mathcal L)) = 1$ if and only if the projective variety $V(F)\subseteq \PP(\mathcal L)$ is one of the following plane curves:
\begin{enumerate}
\item[(a)] A smooth conic.
\item[(b)] The union of three nonconcurrent lines.
\item[(c)] The union of a smooth conic and a tangent line. \qed
\end{enumerate}
\end{theorem}

\begin{theorem}\label{higher-linear}
{\rm\cite[Theorem 4]{Huh14(2)}}.\label{huh-result}
Let $\mathcal L\isom \CC^m$ where $m \geq 4$ and $F$ a homogeneous polynomial on $\mathcal L$. Suppose that all singularities of the projective variety $V(F)\subseteq \PP(\mathcal L)$ are isolated. Then $\myMLD_F(\PP(\mathcal L)) = 1$ if and only if $V(F)$ is a smooth quadric. \qed
\end{theorem}

\begin{remark}
When $F$ is the Fermat quadric, Theorems~\ref{higher-linear} and~\ref{plane-curves} together with Proposition~\ref{prop:EDdegree} yield a classification of linear spaces of Euclidean distance degree one. For instance, a linear subspace of $\CC^m$, where $m\geq 4$, has ED degree one if and only if it intersects the Fermat quadric hypersurface transversally.
\end{remark}

We solve the homaloidal PDE when $\mathcal L = \CC$ or $\CC^2$, without using the assumptions on $F$ in the previous two theorems: \resp{we no longer assume that $F$ has no repeated factors, nor that all singularities of $V(F)$ are isolated}.

\begin{lemma}\label{lem:product-homaloidal}
Let $(p_i)_{i=1}^n$ be a basis of a
$\CC$-vector space $\vectorspace$ and $(\ell_i)_{i=1}^n$ the associated dual basis of $\mathcal L^\ast$. Let $F\coloneqq \prod_{i=1}^n \ell_i^{a_i}$, with $a_i\geq 0$ for all $i$. Then $F$ is homaloidal if and only if all $a_i \geq 1$. In this case, the inverse $\Psi:\vectorspace^\ast\to\vectorspace$ of $\nabla\log F$ is
\[
\Psi(\bu) = \sum_{i=1}^n \frac{a_i}{\bu(p_i)}\,p_i.
\]
\end{lemma}
\begin{proof}
For $p\in \vectorspace$ we have
\[
\nabla_p\log F = \sum_{i=1}^n \frac{a_i}{\ell_i(p)}\,\ell_i,
\]
thus $\Im(\nabla\log F)\subseteq \Span(\ell_i)_{a_i\neq 0}$. So, if $\nabla\log F$ is birational then all $a_i \geq 1$. Conversely, if all $a_i \geq 1$ then the map $\Psi$ of the statement is a birational inverse of $\nabla\log F$.
Indeed, by substituting $\bu = \nabla_p\log F$ in the expression for $\Psi$ we obtain \[(\Psi \circ \nabla\log F)(p) = \sum_{i=1}^n \frac{a_i\, \ell_i(p)}{a_i} p_i = p.\qedhere\]
\end{proof}

\begin{remark}\label{bruno}
Bruno~\cite[Theorem B]{B07} showed the more general result that $\prod_{i=1}^n \ell_i^{a_i}$ 
as in Lemma~\ref{lem:product-homaloidal}, 
with $a_i \geq 1$, is homaloidal if and only if the $\ell_i$ form a basis of $\vectorspace^\ast$.
\end{remark}

\begin{proposition}
\label{prop:univariate_and_bivariate}
Let $F$ be a \resp{nonzero} homogeneous polynomial on $\vectorspace$.
\begin{enumerate}
\item[(a)] If $\vectorspace \isom \CC$, then $\myMLD_{F}(\PP(\vectorspace)) = 1$ if and only if $F$ is not constant. In this case,
\[
\MLE_{\PP(\vectorspace),\vectorspace, F}(\bu) = \frac{\deg(F)}{\bu(p)}\, p
\quad \text{ for general } \bu\in \vectorspace^\ast,
\]
where $p$ is any generator of $\vectorspace$. 
\item[(b)] If $\vectorspace\isom \CC^2$, then $\myMLD_{F}(\PP(\vectorspace)) = 1$ if and only if $F = \lambda\ell_1^{a_1}\ell_2^{a_2}$ for some $\lambda\in\CC$, $a_1,a_2 \geq 1$, and linearly independent $\ell_1,\ell_2\in\vectorspace^\ast$. In this case,
\[
\MLE_{\PP(\vectorspace),\vectorspace,F}(\bu) = \frac{a_1}{\bu(p_1)}\,p_1 + \frac{a_2}{\bu(p_2)}\,p_2
\quad \text{ for general } \bu\in \vectorspace^\ast,
\]
where $p_1,p_2\in\vectorspace$ form a basis dual to $(\ell_1,\ell_2)$. 
\end{enumerate}
\end{proposition}
\begin{proof}
Part (a) is a direct consequence of Lemma~\ref{lem:product-homaloidal}. For (b),
decompose $F$ as a product of powers of $k$ distinct linear forms and use Remark~\ref{bruno} to conclude that $F = \lambda \ell_1^{a_1}\ell_2^{a_2}$ with $\ell_1, \ell_2$ linearly independent. The rest follows from Lemma~\ref{lem:product-homaloidal}.
\end{proof}

\resp{
A point $p$ embedded in a larger projective space $\mathbb{P}^n$ need not have ML degree one. Consider a non-constant polynomial $F$ on $\mathbb{P}^n$.
If $p\notin V(F)$, then $F$ restricted to the affine line spanned by $p$ is nonzero, and thus the ML degree of $p$ is one by Proposition~\ref{prop:univariate_and_bivariate}(a).
Otherwise, if $p \in V(F)$, then the ML degree is zero by Definition~\ref{ml-degree-general}. 
}

\begin{remark}
We comment on a possible way to extend Proposition~\ref{prop:univariate_and_bivariate} beyond the bivariate case.
Let $D_{F,\bu} \coloneqq V(F)\cup V(\bu)$ and let $\chi_{\mathrm{top}}$ denote the topological Euler characteristic. By~\cite[Theorem 1]{DP03}, 
\begin{equation}
\label{eqn:DP03}
    \myMLD_{F}(\PP(\vectorspace)) = (-1)\chi_{\mathrm{top}}(\PP(\vectorspace)\setminus D_{F,\bu})
\quad\text{ for general }\bu\in\vectorspace^\ast.
\end{equation}
When $\mathcal L \isom \CC^2$, the Euler characteristic on the right hand side of~\eqref{eqn:DP03} can be computed as follows.
Consider the factorization of $F$ as a product of powers of $k$ distinct linear forms. 
Then $k = \#V(F)$ and, for general $\bu$, the space $\PP(\vectorspace)\setminus D_{F,\bu}$ is obtained by removing $k + 1$ distinct points from a $2$-sphere. It follows that $\myMLD_{F}(\PP(\vectorspace)) = 1$ if and only if $k=2$, as seen in Proposition~\ref{prop:univariate_and_bivariate}.
For higher-dimensional $\mathcal L$,  the Euler characteristic becomes more difficult to compute.
\end{remark}

\begin{remark}
\label{rem: MLDlinearradical}
Whether $F$ is homaloidal depends only on its radical, thus we may discard repeated factors of $F$. Indeed, the right-hand side of~\eqref{eqn:DP03} only depends on the topology of $V(F)$, so we may replace $F$ with $\rad F$.
In particular, the restriction in Theorem \ref{plane-curves} that $F$ has no repeated factors is not necessary.
\end{remark}

Returning to the case $F = \det$ and $\vectorspace$ embedded into $\Sym(\RR^m)$, the upshot of the results in this section is the following.

\begin{corollary}
    \label{cor: classifyMLD1dimleq3}
    A linear space $\vectorspace \subset \Sym(\RR^m)$ of dimension $k \leq 3$ has ML degree one if and only if the variety $\PP (V(\det) \cap \vectorspace)$ is
    \begin{enumerate}
        \item[\rm{($k$=1)}] Empty. 
        \item[\rm{($k$=2)}] Two distinct points.
        \item[\rm{($k$=3)}] A smooth conic, the union of three nonconcurrent lines, or
the union of a smooth conic and a tangent line. 
    \end{enumerate}
\end{corollary}
\begin{proof}
    For $k=1$ and $k=2$, the statement follows from  Proposition \ref{prop:univariate_and_bivariate} after restricting the determinant to $\vectorspace$. The case $k=3$ is derived similarily using Theorem \ref{plane-curves}.
\end{proof}

\resp{Corollary~\ref{cor: classifyMLD1dimleq3} implies that a linear space represented as the span of one matrix has ML degree one if and only if that matrix is non-singular. A two-dimensional linear space has ML degree one if and only if it can be spanned by two singular matrices that are (up to scaling) the only singular matrices in the space.}

Proposition~\ref{lem:homaloidal-equiv} relates Gaussian models of the simplest type, linear models, to homaloidal polynomials. Such a model has ML degree one if and only if the determinant restricted to it is a homaloidal polynomial. Varying the dimension of the ambient space $\Sym(\mathbb C^m)$, this produces every homogeneous polynomial, by Proposition~\ref{prop:statistical}. Hence classifying all linear models of ML degree one independently of the ambient space is equivalent to classifying all homaloidal polynomials. The latter is a long-standing open problem in birational geometry. For a discussion on the difficulty of this problem, we refer  to~\cite{homaloidal-hypersurfaces}. Even if we fix the ambient space $\Sym(\mathbb C^m)$, homaloidal polynomials arising as the determinant of a linear space are ill-understood~\cite{homaloidal-determinants}.

In Theorems~\ref{plane-curves} and~\ref{higher-linear}, we used the literature on homaloidal polynomials to find Gaussian models of ML degree one. 
One could continue in this fashion, turning results on homaloidal polynomials, such as those in~\cite{homaloidal-hypersurfaces,homaloidal-determinants}, into results on varieties of ML degree one. This would strengthen the bridge between birational geometry and algebraic statistics.

\section{Graphical models} 
\label{sec: graphicalmodels}
In this section, we illustrate our results by applying them to Gaussian graphical models, both directed and undirected. 
We use the homaloidal PDE~\eqref{pde} to find maximum likelihood estimators, which agree with formulae from~\cite{Lau96}. We construct a maximum likelihood estimator from products of determinants of positive definite symmetric matrices 
and compute its image, the corresponding model. Two determinant equations are relevant here: a product of determinants for a chordal undirected graph, from~\cite[Lemma 5.5]{Lau96} 
and its directed analogue, which we derive in Lemma~\ref{determinant-graphical}.
We show that these two determinant equations lead to solutions to the homaloidal PDE. We first explain how to formulate Theorems~\ref{thm:3conditionspsi} and \ref{thm:pde} for $\vectorspace = \Sym(\CC^m)$, without passing to its dual vector space, to simplify our later computations.
    
\subsection{Coordinate formulation of main results}

The bilinear trace pairing $(A,B)\mapsto \tr(AB)$ on $\Sym(\CC^m)$ restricts to an inner product on $\Sym(\RR^m)$.
Using that pairing, we identify $\Sym(\CC^m)$ with its dual vector space via $\Sym(\CC^m) \ni A \mapsto \tr(A\bullet)$.
Under this identification, the gradient $\nabla_S \Phi$ of a scalar-valued function $\Phi$ on $\Sym(\CC^m)$ is a symmetric matrix with entries \begin{equation}
     (\nabla_S \Phi)_{ij} = \frac{1}{2 -\delta_{ij}}\frac{\partial \Phi}{\partial s_{ij}},
\end{equation}
where $\delta_{ij}$ is the Kronecker delta and the $s_{ij}$ are the coordinates on $\Sym(\CC^m)$. \resp{The factors of $1/2$ in the off-zero entries come from the choice of the trace pairing.} Jacobi's formula gives $
        \nabla_S\log \det = S^{-1}
    $.

\begin{example}
    \label{ex:jacobi}
    When $m=2$ we have
    $$
    \det(S) = \det \begin{bmatrix}
        s_{11} & s_{12} \\ s_{12} & s_{22}
    \end{bmatrix} = s_{11}s_{22} - s_{12}^2 , \qquad \nabla_S \log \det = \frac{1}{\det(S)}
    \begin{bmatrix}
    s_{22} & -s_{12} \\
    -s_{12} & s_{11}
    \end{bmatrix}.
    $$
    The off-diagonal entries in the second equation are $\frac12 \frac{\partial \log \det}{\partial s_{12}} = -s_{12}$. 
\end{example}

Maximum likelihood estimation for graphical models fits the setting described in Definition~\ref{ml-degree-general} by specializing to $F = \det$ and $\bu = \mathrm{tr}(S\bullet)$. 
We reformulate Theorems~\ref{thm:3conditionspsi} and ~\ref{thm:pde} for that special case.

\begin{corollary}
    \label{cor:3conditionsSym}
    Let $X$ be an irreducible projective subvariety of $\PP (\Sym(\CC^m))$. We have $\mathrm{MLD}_{\det}(X) = 1$ if and only if there exists a dominant rational map $\Psi: \Sym(\CC^m) \dashrightarrow C_X$
such that
\begin{enumerate}
\item[(a)] $\Psi(t S) = t^{-1} \Psi(S)$ for all $t\in \CC\setminus \{ 0\}$,
\item[(b)] $\nabla_S (\log \det \Psi) = - \Psi(S)$
for general $S\in \Sym(\CC^m)$.
\end{enumerate}
The map
$\Psi$ is the maximum likelihood estimator.
\end{corollary}

\begin{corollary}
    \label{cor:pdeSym}
    There is a bijection between the 
    projective varieties $X\subseteq \PP(\Sym(\CC^m))$ with $\mathrm{MLD}_{\det}(X) = 1$ and the
    solutions $\Phi$ to
\begin{equation}
\label{eq: symPDE}
    \Phi = \det(- \nabla_S \log \Phi),
\quad
\Phi:\Sym(\CC^m) \dashto\mathbb C\textup{ rational and homogeneous.}
\end{equation}
The bijection sends a function $\Phi$ to the variety
$\PP(\overline{\mathrm{Im} (-\nabla_S \log\Phi)})$.
\end{corollary}

The functions $\Phi$ and $\Psi$ from Corollaries~\ref{cor:3conditionsSym} and~\ref{cor:pdeSym} relate via $\Phi = \det \Psi$.

\begin{example}
\label{ex:full_model}
Let the model be the full positive definite cone.
Its Zariski closure is, as a projective variety, $\PP(\Sym(\CC^m))$.
Recall from~\eqref{eqn:lS} that the log-likelihood $\ell_S(K)$ is, up to additive and multiplicative constants, equal to
$\log\det (K) - \mathrm{tr}(KS)$. 
Then
$$\nabla_K \log \det K = K^{-1} \quad \text{and} \quad\nabla_K \mathrm{tr}(KS) = S,$$
where the first equation comes from 
the Jacobi formula, see Example \ref{ex:jacobi}.
Hence the relation
$\nabla \ell_S(K) = 0$ is equivalent to $K^{-1} - S = 0$, leading to the solution $\hat{K} \coloneqq S^{-1}$. 
Thus,
if the sample covariance matrix $S$ has full rank, then $S^{-1}$ is the global maximum and unique critical point of the log-likelihood.  
Hence the maximum likelihood estimator is $\Psi(S)=S^{-1}$ and, 
moreover, we have $\Phi(S)= {\det(S)}^{-1}$. 

We check that $\Psi$ satisfies the conditions of Corollary~\ref{cor:3conditionsSym}, as follows.
The map $\Psi(S) = S^{-1}$ is a dominant rational map and $\Psi(tS)=(tS)^{-1} = t^{-1}S^{-1}=t^{-1}\Psi(S)$. In addition, 
\begin{equation}
\label{eqn:conditionb}
    \nabla_S (\log \det \Psi) = \nabla_S \log \det (S^{-1}) = -\nabla_S \log\det S = -S^{-1} = - \Psi(S).
\end{equation}
Next we show that $\Phi = \det \Psi$ satisfies the conditions of Corollary~\ref{cor:pdeSym}, as follows.
Taking determinants on both sides of~\eqref{eqn:conditionb} and setting $\Phi = \det \Psi$ gives
$\det(-\nabla_S (\log \Phi )) = \Phi$. 
\end{example}

\subsection{Undirected graphical models}

\begin{definition}
\label{def:Lg}
Fix an undirected 
graph $G = (V,E)$. The undirected Gaussian graphical model
$\mathcal M(G)$
consists of (positive definite) concentration matrices $K$ with $K_{ij} = 0$ if $(i,j) \notin E$. 
Its Zariski closure is the linear space
$\vectorspace_G \subseteq \Sym(\CC^V)$. \resp{Here $\CC^V$ denotes the function space that is isomorphic to the vector space over $\CC$ of finite dimension $|V|$.}
\end{definition}

\begin{example}
\label{ex:complete_graph}
    Let $G$ be the complete graph on $m$ nodes. The associated graphical model is the full cone of symmetric positive definite matrices. Its Zariski closure is $\vectorspace_G = \Sym(\CC^m)$. Hence this is Example~\ref{ex:full_model}. 
\end{example}

A fundamental class of undirected graphical models is given by \emph{chordal} graphs, also known as \emph{triangulated} or \emph{decomposable} graphs.

\begin{definition}
Let $G = (V,E)$ be an undirected graph. A \emph{weak decomposition} of $G$ is a triple $A, B, C \subseteq V$ such that
\begin{enumerate}
\item $A, B$ and $C$ are pairwise disjoint with union $V$,
\item $C$ separates $A$ from $B$; \resp{that is, every path from a vertex in $A$ to a vertex in $B$ passes through the set $C$, see~\cite[Chapter 13]{sullivant2018algebraic}.}
\item $C$ induces a complete graph.
\end{enumerate}
We denote such a decomposition by $G = A \amalg_C B$, after identifying a vertex set with its induced subgraph. The notion of a \emph{chordal} graph can be defined recursively by saying that complete graphs are chordal and that a graph $G$ is chordal if it has a weak decomposition $G = A\amalg_C B$ where $A$ and $B$ are chordal.
\end{definition}

An alternative definition of a graph being chordal is that it has no $n$-cycle for $n \geq 4$ as an induced subgraph. See \cite[Sec 2.1]{Lau96} for the equivalence of these definitions. 

Chordal graphs have ML degree one~\cite{SU10}. We give an alternative proof of this result using the homaloidal PDE. 
Given a matrix $S \in \Sym(\CC^V)$ we denote by $S_A \in \Sym(\CC^A)$ its submatrix indexed by $A \subseteq V$.
Dually, given a matrix $M \in \Sym(\CC^A)$, we pad it with zeros at positions in $V \backslash A$ to obtain a $(|V|\times |V|)$-matrix denoted by $[M]^V$.
Define $[S^{-1}_A]^V \coloneqq [(S_A)^{-1}]^V$; i.e., it is the matrix $S_A$, inverted, then padded with zeros.

The following lemma finds the determinant formula we use for finding solutions to the homaloidal PDE for chordal graphical models.

\begin{lemma}\label{determinant-undirected}
Let $A\cup B\cup C$ be a partition of $\{1,\dotsc,m\}$. Let $S$ be an $m \times m$ general positive definite symmetric matrix. Let $\hat K\coloneqq [S^{-1}_{A\cup C}]^V + [S^{-1}_{B\cup C}]^V - [S^{-1}_C]^V$. Then
\[
\det \hat K = \frac{\det(S_C)}{\det(S_{A\cup C})\det(S_{B\cup C})}.
\]
\end{lemma}
\begin{proof}
Let $f$ be the probability density associated to $\hat K^{-1}$ and $g$ the density associated to $S$. Denote by $g_{A\cup C}$ the marginal density associated to $A\cup C$, and likewise define $g_{B\cup C}$ and $g_C$. We compute
\begin{align*}
f(y) &= (2\pi)^{-m/2}(\det \hat K)^{1/2}\frac{\exp(-\frac{1}{2}y^T S^{-1}_{A\cup C} y)\exp(-\frac{1}{2}y^TS^{-1}_{B\cup C}y)}{\exp(-\frac{1}{2}y^TS^{-1}_{C}y)}\\
&= \frac{(\det \hat K)^{1/2}(\det S_C^{-1})^{1/2}}{(\det S^{-1}_{A\cup C})^{1/2}(\det S^{-1}_{B\cup C})^{1/2}}\,\frac{g_{A\cup C}(y) g_{B\cup C}(y)}{g_C(y)}.
\end{align*}
It remains to show that $g_{A\cup C}g_{B\cup C}/g_C$ is a probability distribution. We integrate over $y\in \mathbb R^n$, to obtain
\[
\int_C\int_B\int_A \frac{g_{A\cup C}(y) g_{B\cup C}(y)}{g_C(y)}
= \int_C \frac{1}{g_C(y)} (\int_A g_{A\cup C}(y))(\int_B g_{B\cup C}(y))
= \int_C \frac{g_C(y)^2}{g_C(y)} = 1 .\qedhere
\]
\end{proof}

We proved Lemma~\ref{determinant-undirected} in the real positive definite setting; a proof in the complex setting is outlined in~\cite[Lemma 5.5]{Lau96}. We now describe the solution to the homaloidal PDE for chordal graphical models and the associated MLE map. In particular, we show that such models have ML degree one.

\begin{proposition}\label{prop: chordal-mle}
Let $G = (V,E)$ be a chordal graph. Define the function $\Phi_G$ recursively by
\begin{gather*}
\Phi_G(S) = \begin{cases} \det(S^{-1}) & G \text{ complete,}\\
\dfrac{\Phi_{A\cup C}(S_{A\cup C})\,\Phi_{B\cup C}(S_{B\cup C})}{\Phi_C(S_C)} & G=A\amalg_C B. \\  \end{cases} 
\end{gather*}
Then $\Phi_G$ satisfies the homaloidal PDE. The associated MLE $\Psi_G \coloneqq -\nabla_S\log\Phi_G(S)$ satisfies
\begin{gather*}
\Psi_G(S) = \begin{cases} S^{-1} & G \text{ complete,}\\
 [\Psi_{A\cup C}(S_{A\cup C})]^V + [\Psi_{B\cup C}(S_{B\cup C})]^V - [\Psi_C(S_C)]^{V} & G=A\amalg_C B.
\end{cases} 
\end{gather*}
The function $\Psi_G$ maps surjectively to $\mathcal M(G)$.
\end{proposition}

\begin{proof}The recursion in the statement terminates if and only if $G$ is chordal. If $G$ is complete, the statement holds by Example~\ref{ex:full_model}. We induct on the structure of $G$. If $G$ is chordal but not necessarily complete, let $G = A\amalg_C B$ be a non-trivial weak decomposition. Graphs $A$, $B$, and $C$ are chordal, thus by the induction hypothesis $\Phi_C$, $\Phi_{A\cup C}$, and $\Phi_{B\cup C}$ satisfy the homaloidal PDE and 
the corresponding MLEs $\Psi_C$, $\Psi_{A\cup C}$, and $\Psi_{B\cup C}$ exist and map surjectively onto $\mathcal M(C)$, $\mathcal M(A\cup C)$ and $\mathcal M(B\cup C)$, respectively.
Using the definition of $\Phi_G$, we compute
\begin{align*}
\Psi_G &= -\nabla_S\log\Phi_G(S)
\\ &= -\nabla_S\log\Phi_{A\cup C}(S) - \nabla_S\log\Phi_{B\cup C}(S) + \nabla_S\log\Phi_{C}(S)\\ &= [\Psi_{A\cup C}(S_{A\cup C})]^V + [\Psi_{B\cup C}(S_{B\cup C})]^V - [\Psi_C(S_C)]^{V}.
\end{align*}
Furthermore, by Lemma~\ref{determinant-undirected},
\[\det \Psi_G = \frac{(\det \Psi_C^{-1})}{(\det \Psi_{A\cup C}^{-1})(\det\Psi_{B\cup C}^{-1})} = \frac{\Phi_{A\cup C}\Phi_{B\cup C}}{\Phi_C} = \Phi_G .\] 
Thus $\Phi_G$ satisfies the homaloidal PDE.
Finally, 
let $K\in\mathcal M(G)$. Then the submatrices $K_{A\cup C}$, $K_{B\cup C}$ and $K_{C}$ are elements of $\mathcal M(A\cup C)$, $\mathcal M(B\cup C)$ and $\mathcal M(C)$, respectively. Since $K_{ij} = 0$ whenever $i\in A$ and $j\in B$, we have $K = [K_{A\cup C}]^V + [K_{B\cup C}]^V - [K_C]^V.$ Using the induction hypothesis, we see that $K$ is in the image of $\Psi_G$. Hence, $\Psi_G$ maps surjectively onto $\mathcal M(G)$.
\end{proof}

The previous proposition illustrates how one could start with determinant equations~$\Phi$, verify the homaloidal PDE, and then compute the log-derivative to get an MLE map, whose image is the corresponding model. We explore some examples.

\begin{example}\label{ex: 1-2-3}
Let $G$ be the graph
\[\begin{tikzcd}
	1 & 2 & 3.
	\arrow[no head, from=1-1, to=1-2]
	\arrow[no head, from=1-2, to=1-3]
\end{tikzcd}\]
The graphical model $\mathcal M(G)$ consists of all concentration matrices 
\[ K = \begin{bmatrix} 
k_{11} & k_{12} & 0 \\ k_{12} & k_{22} & k_{23} \\ 0 & k_{23} & k_{33}
\end{bmatrix}.
\]
Following the recursion, we find that the solution to the homaloidal PDE for $\mathcal M(G)$ is 
\[
\Phi(S)=\frac{\det(S_2)}{\det(S_{12})\det(S_{23})}.
\]
We can compute the MLE by taking partial derivatives of $\Phi$. For example, the $(2,2)$ entry of the MLE is
\[
\mathrm{MLE(S)}_{22} = -\frac{\partial \log \Phi}{\partial s_{22}} = \frac{ (s_{11}s_{22}^2s_{33}-s_{12}^2s_{23}^2)}{ (s_{22})(s_{22}s_{33}- s_{23}^2)(s_{11}s_{22}-s_{12}^2)} .
\]
\end{example}

\begin{example}
Consider the graphical model defined by the graph
\[\begin{tikzcd}
	& 3 \\
	1 & 2 & 4 & 5.
	\arrow[no head, from=2-1, to=2-2]
	\arrow[no head, from=2-2, to=2-3]
	\arrow[no head, from=1-2, to=2-2]
	\arrow[no head, from=1-2, to=2-3]
	\arrow[no head, from=2-3, to=2-4]
\end{tikzcd}\]
Following the recursion, the solution to the homaloidal PDE for this graphical model is 
\[
\Phi(S)=\frac{\det(S_2)\det(S_4)}{\det(S_{12})\det(S_{234})\det(S_{45})}.
\]
Again, partial derivatives of $\Phi$ give the MLE. For instance,
\[
\mathrm{MLE(S)}_{23} = -\frac{1}{2}\frac{\partial \log \Phi}{\partial s_{23}} = \frac{1}{2}\frac{\partial \log \det(S_{234})}{\partial s_{23}} = \frac{ s_{24}s_{34}-s_{23}s_{44}}{ \det(S_{234})}.
\]
\end{example}

\begin{example}[The four cycle]\label{undirected-four-cycle}
Let $G$ be the non-chordal graph
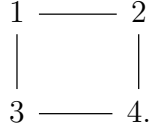
\[\begin{tikzcd}
	1 & 2 \\ 3 & 4.
	\arrow[no head, from=1-1, to=1-2]
	\arrow[no head, from=1-2, to=2-2]
    \arrow[no head, from=2-1, to=2-2]
    \arrow[no head, from=1-1, to=2-1]
\end{tikzcd}\]
The naive analogue for the determinant formula for this graph has separators in the numerator and cliques in the denominator:
\[
\Phi(S)=\frac{\det(S_1)\det(S_2)\det(S_3)\det(S_4)}{\det(S_{12})\det(S_{23})\det(S_{34})\det(S_{14})}.
\]
This formula does not satisfy the homaloidal PDE. To verify this, we
enter this expression
into the computer algebra system \texttt{Macaulay2}~\cite{M2}, compute $\det(-\nabla\log \Phi)$ symbolically, and find that it is not equal to $\Phi$. See~ \url{https://mathrepo.mis.mpg.de/GaussianMLDeg1} for code to verify this and the other examples in this section.
\end{example}

\subsection{Directed graphical models}

\begin{definition}
Fix a directed acyclic graph (DAG) $G = (V,E)$. A \emph{directed Gaussian graphical model} on $G$ consists of concentration matrices $K = (I- A)^\top \Omega (I - A)$, where $\Omega$ is diagonal positive definite and $A_{ij} = 0$ unless $j \to i$ is in $E$. 
\end{definition}

We prove that directed Gaussian graphical models have ML degree one by providing a solution to the homaloidal PDE,
in Proposition \ref{prop:directed_phi}.
We begin with the determinant formula that is analogous to Lemma \ref{determinant-undirected}.
The formula makes use of the \emph{parents} of a node $v \in V$; i.e.,  $\pa(v) \coloneqq \{ i \mid i \to v ~\textnormal{in}~E \}$.

\begin{lemma}\label{determinant-graphical}
Let $G = (V,E)$ be a DAG and $S$ a general positive definite symmetric  matrix of size $|V| \times |V|$. Define
\begin{equation}
    \label{eqn:MLE_for_K_DAG}
    \hat K \coloneqq \sum_{v\in V} K_{[v|\pa(v)]}, \quad \text{where} \quad K_{[v|\pa(v)]}\coloneqq [S^{-1}_{v\cup\pa(v)}]^V - [S^{-1}_{\pa(v)}]^V.
\end{equation}
Then
\[
\det \hat K = \prod_{v\in V}\frac{\det S_{ \pa(v)}}{\det S_{v\cup\pa(v)}}.
\]
\end{lemma}

\begin{proof}
The density function associated to $\hat K^{-1}$ is
\begin{align*}
f(y) &= (2\pi)^{-m/2}(\det \hat K)^{1/2}\exp\left(-\frac{1}{2}y^T \hat K y\right) \\
&= (2\pi)^{-m/2}(\det \hat K)^{1/2}\prod_{v\in V}\exp\left(-\frac{1}{2}y^T K_{[v|\pa(v)]} y\right).
\end{align*}
We obtain a second description of $f(y)$. We have
\[
f(y_v|y_{\pa(v)})
= \frac{f(y_{v\cup \pa(v)})}{f(y_{\pa(v)})}
= (2\pi)^{-1/2}\frac{\det(S_{\pa(v)})^{1/2}}{\det(S_{v\cup\pa(v)})^{1/2}} \exp\left(-\frac{1}{2}y^TK_{[v|\pa(v)]}y\right).
\]
Thus,
\[
f(y) =
(\det \hat K)^{1/2}\bigg(\prod_{v\in V}\frac{\det S_{v\cup \pa(v)} }{\det{S_{\pa(v)}}}\bigg)^{1/2} \bigg(\prod_{v\in V} f(y_v|y_{\pa(v)})\bigg).
\]
According to the factorization property of $G$, the rightmost factor of the above expression is a probability distribution. Integrating over $y$ we obtain the desired result.
\end{proof}

While we prove Lemma~\ref{determinant-graphical} in the real positive definite setting, a proof can also be found in the general complex setting, by following the structural equations of the graph~$G$ with complex variables.

\begin{proposition}
\label{prop:directed_phi}
Let $G = (V,E)$ be a DAG and $\mathcal M$ its associated Gaussian graphical model. The function
\begin{equation}
\label{eqn:dag_det}
    \Phi(S)\coloneqq \prod_{v\in V}\frac{\det(S_{\pa(v)})}{\det(S_{v\cup \pa(v)})}
\end{equation}
satisfies the homaloidal PDE. The corresponding MLE $\Psi$ sends $S$ to $\hat K$ and maps surjectively to $\mathcal M$.
In particular, every DAG model has ML degree one.
\end{proposition}
\begin{proof}
We have
\begin{align*}
-\nabla_S\log\Phi
&= \sum_{v\in V}\big(\nabla_S\log\det(S_{v\cup\pa(v)}) - \nabla_S\log\det(S_{[\pa(v)]})\big)\\
&= \sum_{v\in V}\big([S_{v\cup \pa(v)}^{-1}]^V - [S_{\pa(v)}^{-1}]^V\big) = \hat K
\end{align*}
and $\det(-\nabla\log\Phi) = \Phi$, by Lemma~\ref{determinant-graphical}. In the proof of the lemma, we also see that $\hat K \in \mathcal M$ since the distribution associated to $\hat K$ factorizes according to $G$.
\end{proof}

\begin{example}
\label{ex: unshieldedcollider}
Fix the DAG $1\rightarrow 3 \leftarrow 2$. 
Consider the directed Gaussian graphical model on $G$.
The model consists of concentration matrices $K$ that satisfy 
$$ k_{13} k_{23} - k_{12} k_{33} = 0.$$
Let $S$ be the sample covariance matrix. 
The MLE given $S$ is 
$$\Psi(S) = \frac{1}{\det(S)} \begin{bmatrix} 
\frac{s_{11}^2s_{22}^2s_{33}+\dots-2s_{11}s_{12}^2s_{22}s_{33}}{s_{11}(s_{11}s_{22}-s_{12}^2)} & \frac{(s_{12}s_{13}-s_{11}s_{23} )(s_{12}s_{23}-s_{13}s_{22})}{s_{11}s_{22}-s_{12}^2}  & s_{12}s_{23} - s_{13}s_{22} \\
\frac{(s_{12}s_{13}-s_{11}s_{23} )(s_{12}s_{23}-s_{13}s_{22})}{s_{11}s_{22}-s_{12}^2}     & \frac{s_{11}^2s_{22}^2s_{33}+\dots-2s_{11}s_{12}^2s_{22}s_{33}}{s_{22}(s_{11}s_{22}-s_{12}^2)}  & s_{12}s_{13} - s_{11}s_{23}  \\
s_{12}s_{23} - s_{13}s_{22}  &  s_{12}s_{13} - s_{11}s_{23}   & s_{11}s_{22}-s_{12}^2 \\
\end{bmatrix}, $$
provided that $\Delta\coloneqq s_{11}s_{22}(s_{11}s_{22}-s_{12}^2)\det(S)$ does not vanish.
Here $$\Phi = \det \Psi = \frac{s_{11}s_{22}-s_{12}^2}{s_{11}s_{22}\det(S)} .$$   
The formula for the scalar valued function $\Phi$ is simpler than the one for the MLE $\Psi$.
\end{example}

\begin{example}
\label{ex:directed_bigger}
Let $G$ be the graph
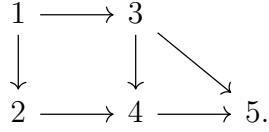
\[\begin{tikzcd}
	1 & 3 \\
	2 & 4 & 5.
	\arrow[from=1-1, to=1-2]
	\arrow[from=1-1, to=2-1]
	\arrow[from=1-2, to=2-2]
	\arrow[from=2-1, to=2-2]
	\arrow[from=1-2, to=2-3]
	\arrow[from=2-2, to=2-3]
\end{tikzcd}\]
The solution to the homaloidal PDE for $G$ is
\[
\Phi(S) = \frac{\det(S_{1})\det(S_{23})\det(S_{34})}
{\det(S_{12})\det(S_{13})\det(S_{234})\det(S_{345})}.
\]
Computing partial derivatives of $\Phi$ gives the MLE, $\hat{K}$. For example, 
\[
\mathrm{MLE(S)}_{23} = -\frac{1}{2}\frac{\partial \log \Phi}{\partial s_{23}} = \frac{1}{2}\frac{\partial \log \nicefrac{\det(S_{234})}{\det(S_{23})}}{\partial s_{23}} = -\frac{(s_{33}s_{24}-s_{23}s_{34})(s_{23}s_{24}-s_{22}s_{34})}{\det(S_{234})\det(S_{23})}.      
\]
The expression for $\Phi$ may not seem simpler than the one for $\hat{K}$ in~\eqref{eqn:MLE_for_K_DAG} which, for this example, is
\[
\hat K = ([S^{-1}_{12}]^V + [S^{-1}_{13}]^V + [S^{-1}_{234}]^V + [S^{-1}_{{345}}]^V)
- ([S^{-1}_{1}]^V + [S^{-1}_{23}]^V + [S^{-1}_{34}]^V).
\]
However, the former expression has smaller complexity (and is faster to compute) than the latter. If $a_v$ is the size of the matrix $S_{\pa(v)}$, then the complexity of the expression $\Phi$ in terms of $S$ is $O(a_1^2 + \dotsb + a_m^2)$, whereas the complexity of $\hat K$ is $O(a_1^3 + \dotsb + a_m^3)$. 
\end{example}

\begin{example}
Let $G$ be the directed non-acyclic graph
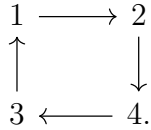
\[\begin{tikzcd}
	1 & 2 \\ 3 & 4.
	\arrow[from=1-1, to=1-2]
	\arrow[from=1-2, to=2-2]
    \arrow[from=2-2, to=2-1]
    \arrow[from=2-1, to=1-1]
\end{tikzcd}\]
Applying the formula~\eqref{eqn:dag_det} for $\Phi$ to this example gives an equation that does not satisfy the homaloidal PDE: it is the same as Example~\ref{undirected-four-cycle}.
\end{example}

\section{Solutions to the homaloidal PDE}
\label{sec:solutions} 

We saw instances where linear ML degree one varieties can be characterized, via connections to homaloidal polynomials, in Section~\ref{sec:linear_and_homaloidal}. 
In this section, we
present steps towards parametrizing the solutions to the homaloidal PDE when $X$ is not necessarily linear.
We study factorization properties of solutions to the homaloidal PDE in Proposition~\ref{prop:factors}, inspired by the proof of \cite[Lemma 16]{Huh14(1)}.
We solve the PDE when the polynomial $F$ is linear in Theorem~\ref{thm:linear_case}.

\begin{proposition}
\label{prop:factors} 
Fix a basis $u_1,\ldots,u_{\dim \vectorspace}$ for $\vectorspace^\ast$. Let $\Psi$ be the MLE map of some variety with ML degree one. Consider prime decompositions 
of the coordinates of $\Psi$, say
$\psi_i = c_i\prod_{f \in \mathcal{F}} f^{\alpha_{i,f}}$
where  $c_i \in \CC$ and
 $\mathcal{F} $ is the set of all prime factors that appear in some coordinate.
The prime decompositions satisfy:
     \begin{enumerate}[label=(\alph*)]
         \item The factors in the denominators are linear; i.e., $\alpha_{i,f} \geq -1$ for all $f\in\mathcal{F}$ and all $i\in\{1,\ldots, \dim \vectorspace \}$. 
         \item If $f$ appears in some denominator, then $\frac{\partial f}{\partial u_j} \neq 0$ if and only if $\alpha_{j,f} = -1$. That is, $f$ only depends on the $u_j$ for which $f$ is a factor of the denominator of $\psi_j$.
     \end{enumerate}
 \end{proposition}
 \begin{proof}
     There is a rational function $\Phi$ such that $\Psi = -\nabla \log \Phi$,
     by Theorem \ref{thm:pde}.
     We decompose $\Phi = \prod_{k=1}^N g_k^{\beta_k}$ into prime factors for some integer $N$. Fixing a basis, define $\mathcal{G}_i \coloneqq \{k: \, \frac{\partial g_k}{\partial u_i} \neq 0  \}$, the set of indices $k$ such that $g_k$ depends on $u_i$. Then
     \begin{align*}
         \psi_i = -\frac{\partial \log \Phi}{\partial u_i} = -\sum_{k \in \mathcal{G}_i} \frac{\beta_k}{g_k} \frac{\partial g_k}{\partial u_i} =
         -\frac{  \sum_{k\in \mathcal{G}_i} (\prod_{j \in \mathcal{G}_i \setminus \{k\} } g_j) \beta_k \frac{\partial g_k}{\partial u_i}   }{\prod_{k \in \mathcal{G}_i} g_k}.
     \end{align*}
     Then (a) follows, because no exponent in the denominator of any $\psi_i$ can be greater than one after simplifying the expression above. We next show that there are no common prime factors of the numerator and denominator of the derived expression for $\psi_i$. For each index $l \in \mathcal{G}_i$ we may write
     \begin{align*}
            \sum_{k\in \mathcal{G}_i} \left( \prod_{j \in \mathcal{G}_i \setminus \{k\}} g_j\right) \beta_k \frac{\partial g_k}{\partial u_i} = \left(\prod_{j \in \mathcal{G}_i\setminus\{l\} } g_j\right) b_l \frac{\partial g_l}{\partial u_i} + g_l\cdot\left(\sum_{k\in \mathcal{G}_i\setminus\{l\}} \left(\prod_{j \in \mathcal{G}_i\setminus\{k,l\} } g_j\right) \beta_k \frac{\partial g_k}{\partial u_i}\right).
     \end{align*}
     From this we deduce that 
     \begin{align*}
        g_{l} \Big| \sum_{k\in \mathcal{G}_i} \left(\prod_{j \in \mathcal{G}_i \setminus \{k\} } g_j\right) \beta_k \frac{\partial g_k}{\partial u_i} \quad\text{if and only if}\quad g_{l} \Big| \left(\prod_{j \in \mathcal{G}_i\setminus\{l\} } g_j\right) b_l \frac{\partial g_l}{\partial u_i}.
     \end{align*}
     However, $g_{l} \nmid (\prod_{j \in \mathcal{G}_i\setminus\{l\} } g_j) b_l \frac{\partial g_l}{\partial u_i}$ because all $g_k$ are distinct primes and a polynomial cannot be a factor of its own derivative. Thus the expression for $\psi_i$ has no common prime factor in the numerator and denominator. From this we conclude that $\mathcal{G}_i \subset \mathcal{F}$ and that these factors $g_k$ are the only ones that occur in the denominator of $\psi_i$. Moreover, $g_k$ will appear exactly in the denominators of $\psi_i$ with $k \in \mathcal{G}_i$. 
 \end{proof}

We can observe the described properties of the map $\Psi$ in all examples of ML degree one varieties presented throughout this paper. We conclude this section by solving the PDE when the polynomial $F$ is linear.
If $Y\subseteq \PP^n$ is a hypersurface and $p\in Y$ a singular point with multiplicity $\deg(Y)-1$, we define a retraction $r_p:\PP^n\dashto Y$ that sends a point $q$ to the unique point in $Y$ on the line through $p$ and $q$. The map $r_p$ is rational.

\begin{theorem}
\label{thm:linear_case}
Let $\vectorspace$ be a
finite-dimensional
$\CC$-vector space and $\ell \in\vectorspace^\ast$.
Let $\Phi = \nicefrac{f}{g}$ be a rational function on $\mathcal L^\ast$, homogeneous of degree $-1$. Then $\Phi$ satisfies the homaloidal PDE with respect to $\ell$ if and only if $\ell$ is a point of $V(g)$ of multiplicity $\deg g - 1$ and $f = \ell(\nabla_{\bu} g)$. 

In this case, the map $\Psi \coloneqq -\nabla\log \Phi$ satisfies $[\Psi(\bu)] = [\nabla_{r_\ell(\bu)}g]$ for general $\bu\in\mathcal L^\ast$, where $r_\ell$ is the retraction to $V(g)$. The variety of ML degree one, parametrized by $\Psi$, is the dual variety $\PP(V(g))^\vee \coloneqq \PP(\overline{\Im(\nabla g|_{V(g)}}).$
\end{theorem}
\begin{proof}
The function $\Phi$ satisfies the homaloidal PDE if and only if
\begin{equation}\label{directional-derivative}
1 = \frac{\ell(-\nabla\log\Phi)}{\Phi(\bu)} = -\frac{\ell( \nabla_{\bu} \Phi)}{\Phi(\bu)^2} 
 = \left[\frac{d}{dt}\frac{1}{\Phi(\bu + t\ell)}\right]_{t=0}.
\end{equation}
Pick a basis $(\ell_1,\dotsc,\ell_n)$ of $\mathcal L^{\ast}$ such that $\ell=\ell_1$. In this basis, elements of $\mathcal L^{\ast}$ are written as $\bu=\sum_{i=1}^n u_i \ell_i$ and \eqref{directional-derivative} is equivalent to
\[
\frac{d}{du_1}\frac{1}{\Phi(\bu)} = 1.
\]
This equation is satisfied if and only if
\[
\frac{1}{\Phi(u)} = u_1 + \frac{h_2}{h_1} = \frac{u_1h_1 + h_2}{h_1},
\]
where $h_1,h_2\in \CC[u_2,\dotsc, u_n]$ and $\deg h_2 = \deg h_1+1$.
The latter is equivalent to
\begin{equation}\label{tangent-cone}
f \in \mathbb C[u_2,\dotsc,u_n] \quad \text{and} \quad
g = u_1 f + h \quad
\text{for some} \quad h\in \CC[u_2,\dotsc,u_n]_{\deg f+1}.
\end{equation}
Since $\ell$ is the origin in the affine chart $u_1=1$, \eqref{tangent-cone} holds if and only if 
$\ell$ is a point of $V(g)$ of multiplicity $\deg g - 1$ with $f = \ell(\nabla_{\bu} g)$.

Let $\Phi$ as above satisfy the homaloidal PDE. Use the same basis as above, so that $\ell = \ell_1$ and $g = u_1 f + h$ where $f,g\in \CC[u_2,\dotsc,u_n]$. Then
\begin{align*}
[\Psi] &= [\nabla\Phi] = [\nabla\log f - \nabla\log(u_1 f + h)] \\ &= \left[\frac{h\nabla f - f^2\nabla u_1 - f\nabla h}{f(u_1 f + h)}\right] = [h\nabla f - f^2\nabla u_1 - f\nabla h].
\end{align*}
The right-hand side does not depend on $u_1$. Hence the function $[\Psi]$ is constant along the line through $\bu$ and $\ell$. Moreover,
\[
[\Psi] = \left[\frac{g\nabla f - f\nabla g}{g^2}\right] = [g\nabla f - f\nabla g],
\]
thus $[\Psi]$ is equivalent to a rational function that is defined on $V(g)$. For general $\bu\in\mathcal L^{\ast}$ we have $g(r_\ell(\bu)) = 0$ and $f(r_\ell (\bu)) \neq 0$, thus
\[
[(g\nabla f - f\nabla g)(r_\ell (\bu))] = [f(r_\ell (\bu)) \nabla_{r_\ell (\bu)}g]
= [\nabla_{r_\ell (\bu)}g].
\]
Since $r_\ell (\bu)$ is on the line through $\bu$ and $\ell$, we conclude that $[\Phi(\bu)] = [\nabla_{r_\ell (\bu)} g]$.
\end{proof}

\section{Examples of ML degree one varieties}
\label{sec:constructing} 

We conclude the paper with examples of ML degree one varieties; see Table~\ref{table:pde-examples} for an overview. \resp{In some of these examples, we take the ML degree with respect to polynomials $F$ other than the determinant, but each can be turned into an example of a variety of ML degree one in the sense of~\cite{SU10}, by applying Proposition~\ref{prop:statistical}}. The first two rows of the table are graphical models examples from Section~\ref{sec: graphicalmodels}. We give one concrete example for each remaining row of Table~\ref{table:pde-examples}.

The next three examples involve a smooth quadric $Q$. After fixing a basis, we write 
 $Q(x) = x^\top A x$, where $x$ is a column vector and $A$ is an invertible symmetric matrix.
For vectors $x$ and $y$, we define $A(x,y) \coloneqq x^\top Ay$, so $Q(x) = A(x,x)$. 
The dual quadric $Q^\vee$ is defined on the dual vector space. With respect to our choice of basis, we have $Q^\vee(\bu)=\bu A^{-1}\bu^\top$, where $\bu$ is a row vector.

Note that $\nabla_x Q = 2x^\top A$ and $\nabla_{\bu} Q^\vee = 2A^{-1}\bu^\top$.
The points on $Q$ and $Q^\vee$ relate via 
\begin{equation}
    \label{eqn:relation_1}
    Q(\nabla_{\bu} Q^\vee) = 4Q^\vee(\bu).
\end{equation}
The gradients of $Q$ and $Q^\vee$ have the following symmetry
\begin{equation}
    \label{eqn:relation_3}
(\nabla_x Q)(y) = 2 A(x,y) \qquad \boldsymbol{v}(\nabla_{\bu} Q) = 2A^{-1}(\boldsymbol{v}, \bu).
\end{equation}
The ambient vector space and its dual relate via $Q$ and $Q^\vee$, since
\begin{equation*}
    A(\nabla_{\bu} Q^\vee, x) = 2\bu(x),
\end{equation*}
where $\bu (x)$ is the multiplication $\bu x$ of row vector $\bu$ with column vector $x$. 

We have the following biduality. Take $\ell \in V(Q^\vee)$. 
Define $p$ to cut out the tangent hyperplane of $Q^\vee$ at $\ell$; that is, define $p \coloneqq \nabla_{\ell} Q^\vee$.
Then $p \in V(Q)$, by~\eqref{eqn:relation_1}. 
We can compute $\ell = \frac14 \nabla_{p} Q$, and hence $\ell$ cuts out the tangent hyperplane of $Q$ at $p$.

The following example is the degree-two case of Theorem \ref{thm:linear_case}.

\begin{example}\label{ex:specialquadrics}
The statement of Theorem~\ref{thm:linear_case} involves a rational function $\Phi = \frac{f}{g}$. Set $g := Q^\vee$. 
Then $X=V(Q)$. 
Let $\ell$ be such that $Q^\vee(\ell) = 0$. This setup satisfies the hypotheses of Theorem~\ref{thm:linear_case}, since $Q^\vee$ is degree two and $\ell$ is a smooth point on $Q^\vee$; i.e., a point of multiplicity one. According to the theorem, $Q$ is ML degree one with respect to $\ell$, and the denominator of the corresponding solution to the homaloidal PDE is $g$, while the numerator is  $f =\ell(\nabla_{\bu} Q^{\vee})$. We simplify the expression of $f$ using the equations in \eqref{eqn:relation_3}, concluding that $f = \bu(\nabla_\ell Q^\vee) = \bu(p)$, where $p  := \nabla_\ell Q^\vee$.
Hence $\myMLD_\ell(X) = 1$ and the solution to the homaloidal PDE is 
$$
\Phi(\bu) = \frac{\bu(p)}{Q^\vee(\bu)}.
$$
As a concrete instance, let $Q^\vee(u_0, u_1,u_2,u_3) = u_0u_1 - u_2u_3$ and $\ell = \begin{bsmallmatrix}1 
 &0& 0&0\end{bsmallmatrix}$. We have $\nabla_{\bu} Q^{\vee} = \begin{bsmallmatrix}u_1 \\ u_0 \\ -u_3 \\ - u_2 \end{bsmallmatrix}$ and $p = \nabla_{\ell} Q^{\vee} = \begin{bsmallmatrix}0 \\ 1 \\ 0 \\ 0 \end{bsmallmatrix}$. Then 
$
\Phi(\bu) = \frac{u_1}{u_0u_1 - u_2u_3}
$,
with the associated variety cut out by $Q(x) = x_0x_1 - x_2x_3$. We compute
$$
\mathrm{MLE}(\bu) = -\nabla_{\bu} \log \Phi = \frac{1}{u_0u_1 - u_2u_3}\begin{bsmallmatrix}u_1 \\ u_0 \\ -u_3 \\ - u_2 \end{bsmallmatrix} - \frac{1}{u_1}\begin{bsmallmatrix}0 \\ 1 \\ 0 \\ 0 \end{bsmallmatrix}.
$$
\end{example}

The next example verifies that all smooth quadrics are homaloidal polynomials, as seen in Theorem~\ref{higher-linear}. We compute the associated solutions to the homaloidal PDE.

\begin{example} 
 \label{ex: smoothquadric}
    We show that $\mathrm{MLD}_Q(\PP^n) =1$ and that the solution to the homaloidal PDE is the reciprocal of the dual quadric 
    $$
    \Phi(\bu) = \frac{4}{Q^\vee(\bu)}.
    $$
This is a consequence of the fact that $Q(\nabla_{\bu} Q^\vee) = 4Q^\vee(\bu)$, which then implies
$$
Q(-\nabla \log \Phi) = Q(-\nabla_{\bu} \log \frac{4}{Q^\vee}) = \frac{ Q(\nabla_{\bu} Q^\vee) }{(Q^\vee)^2 }= \frac{ 4}{Q^\vee }.$$
Hence $\Phi(\bu)$ solves the homaloidal PDE. Moreover, 
$$
\mathrm{MLE}(\bu) = -\nabla_{\bu} \log \Phi = \frac{\nabla_{\bu} Q^\vee}{Q^\vee(\bu)}.
$$
The associated variety $X$ of $\Phi$ is $\PP^n$ because, as described in Corollary \ref{cor:pdeSym}, 
$$
X = \overline{\mathrm{Im}(-\nabla \log Q^\vee) }= \overline{\mathrm{Im}(\nabla Q^\vee) }
$$
and $\nabla Q^\vee$ is an invertible linear map, which is dominant. 
\end{example}

Our third example verifies the fact that a product $Q\cdot \ell$ of a smooth quadric $Q$ and a linear form $\ell$, with $V(\ell)$ tangent to $V(Q)$, is a homaloidal polynomial. We compute the associated solutions to the homaloidal PDE.

\begin{example}
\label{ex:conictangent} 
Let $\dim \vectorspace = n+1$. Let $F = Q \cdot \ell$ be the product of a degree 2 polynomial $Q$ and the equation of its tangent hyperplane $\ell$ at a point $p  = \nabla_\ell Q^\vee$. We show that $F$ is homaloidal. This extends Theorem~\ref{plane-curves}(c), which is the case $n=2$. We show that the solution to the homaloidal PDE is
$$
\Phi(\bu) = 16\frac{\bu(p)}{Q^\vee(\bu)^2}
$$
where $Q^\vee$ is the dual quadric and $p$ the point of tangency. Let $Q(x) = x^T A x$. With this we verify the first part of the suggested solution
\begin{align*}
    - \nabla \log \Phi & =  2 \frac{\nabla_{\bu}  Q^\vee}{Q^\vee} - \frac{p}{\bu(p)}  \quad (= \mathrm{MLE}(\bu)).  \\
    Q(-\nabla \log \Phi) & = Q(2 \frac{\nabla_{\bu}  Q^\vee}{Q^\vee} - \frac{p}{\bu(p)} ) =  
    Q(2\frac{\nabla_{\bu}  Q^\vee}{Q^\vee} ) - 2 A(2 \frac{\nabla_{\bu}  Q^\vee}{Q^\vee}, \frac{p}{\bu(p)} ) + Q(\frac{p}{\bu(p)} ) \\ 
    & = 16\frac{Q^\vee}{(Q^\vee)^2} -  8\frac{\bu(p)}{Q^\vee \bu(p)} + 0 = \frac{8}{Q^\vee}.
\end{align*}
Moreover  we use \eqref{eqn:relation_3}, $\ell(\nabla_{\bu} Q^{\vee}) =\bu(\nabla_{\ell} Q^{\vee}) $, which yields 
\begin{align*}
    \ell(- \nabla \log \Phi) = \ell(2 \frac{\nabla_{\bu}  Q^\vee}{Q^\vee} - \frac{p}{\bu(p)}) =  2\ell( \frac{\nabla_{\bu}  Q^\vee}{Q^\vee} ) - 0 = 2\bu( \frac{\nabla_{\ell}  Q^\vee}{Q^\vee} ) =  \frac{2\bu(p)}{Q^\vee}. 
\end{align*}
To conclude, $\Phi(\bu) =  16\frac{\bu(p)}{Q^\vee(\bu)^2}$ yields a solution to the homaloidal PDE, since
$$
Q\cdot\ell(-\nabla_{\bu} \Phi) = \frac{8}{Q^\vee} \frac{2\bu(p)}{Q^\vee} = \Phi.
$$
We show that the image of $\nabla \log \Phi$ is all of $\vectorspace$, i.e., that the map is dominant, by finding its inverse. More specifically, we show that $-\nabla \log \Phi = (\nabla \log F)^{-1}$. We compute
$$
 \nabla_{(-\nabla_{\bu} \log \Phi)} \log F = \frac{\ell}{\ell(-\nabla_{\bu} \log \Phi)} + \frac{1}{Q(-\nabla_{\bu} \log \Phi)} \nabla_{(-\nabla_{\bu} \log \Phi)} Q.
$$
We expand these expressions to obtain
\begin{align*}
\ell(-\nabla_{\bu} \log \Phi) &= \ell(2 \frac{\nabla_{\bu}  Q^\vee}{Q^\vee(\bu)} - \frac{p}{\bu(p)}) = \frac{2\ell( \nabla_{\bu}  Q^\vee)}{Q^\vee(\bu) } = \frac{2\bu(p)}{Q^\vee(\bu) };  \\
Q(-\nabla_{\bu} \log \Phi) &=  Q(2 \frac{\nabla_{\bu}  Q^\vee}{Q^\vee(\bu)} - \frac{p}{\bu(p)}) = Q(2 \frac{\nabla_{\bu}  Q^\vee}{Q^\vee(\bu)}) -2A(2 \frac{\nabla_{\bu}  Q^\vee}{Q^\vee(\bu)}, \frac{p}{\bu(p)}) + 0  \\
&= \frac{16}{Q^\vee(\bu)} -\frac{8}{\bu(p) \cdot Q^\vee(\bu)}\ell(\nabla_{\bu}  Q^\vee) = \frac{8}{Q^\vee(\bu)};\\
 \nabla_{(-\nabla_{\bu} \log \Phi)} Q &=  \nabla_{(2 \frac{\nabla_{\bu}  Q^\vee}{Q^\vee(\bu)} - \frac{p}{\bu(p)})} Q =  \frac{2}{Q^\vee(\bu)}\nabla_{(\nabla_{\bu}  Q^\vee)} Q - \frac{1}{\bu(p)}\nabla_{p} Q = 8\frac{\bu}{Q^\vee(\bu)} - \frac{4}{\bu(p)}\ell,
\end{align*}
where we have used that $\nabla_xQ$ is linear in $x$ and that $\nabla_{\nabla_{\bu} Q^\vee} Q = 4\bu$. Thus
\begin{align*} 
 \nabla_{(-\nabla_{\bu} \log \Phi)} \log F = \frac{Q^\vee(\bu)}{2\bu(p)}\ell +  \frac{Q^\vee(\bu)}{8}   (8\frac{\bu}{Q^\vee(\bu)} - \frac{4}{\bu(p)}\ell)  = \bu. 
\end{align*}
\end{example}

\begin{example}\label{ex:squareandcorner}
Let $X = V(k_{23},k_{13},k_{12}^2-k_{11}k_{22}+k_{11}k_{33})$. This is the reciprocal variety of the linear space \[ \begin{pmatrix}x&y&0\\y&z&0\\0&0&z\end{pmatrix} .\] 
This is an instance of a \emph{colored covariance graphical model}; it is the graph with three nodes that has a single edge $1 - 2$ and nodes $2$ and $3$ with the same color.
We can show that it has ML degree one, by computing $\Phi$ and $\Psi$ directly.     This can be verified with a computer algebra system, such as
\texttt{Macaulay2}~\cite{M2}. Code for this and the next example can be found in \url{https://mathrepo.mis.mpg.de/GaussianMLDeg1}. They are
\begin{align*}
    \Phi(S) &= \frac{4 s_{22}}{(s_{22} + s_{33} )^2( s_{11}s_{22}-s_{12}^2) } \\
    \Psi(S) &= \frac{-1}{s_{22}}\begin{pmatrix}0&0&0\\0&1&0\\0&0&0\end{pmatrix} + \frac{2}{s_{22} + s_{33}}\begin{pmatrix}0&0&0\\0&1&0\\0&0&1\end{pmatrix} + \frac{1}{s_{11}s_{22} - s_{12}^2} \begin{pmatrix}s_{22}&-s_{12}&0\\-s_{12}&s_{11}&0\\0&0&0\end{pmatrix}
\end{align*}
\end{example}

It would be interesting to generalize Example~\ref{ex:squareandcorner} to general reciprocal linear spaces. This is the missing entry in Table~\ref{table:pde-examples}.

\begin{example}
\label{ex: hyperplane} 
    Consider the hyperplane $X = V(k_{11} - k_{22}) \subset \Sym(\CC^3)$. This is the colored graphical model corresponding to the undirected 3-cycle with first two nodes having the same color.  It has ML degree one with
\[
\Phi(S) = \frac{4s_{33}}{\det((gSg^{\top})_{23})\det ((gSg^{\top})_{13})} \quad
\text{where}\quad
g = \begin{pmatrix}
1&1&0\\
1&-1&0\\
0&0&1/2\\
\end{pmatrix}.
\]
    \end{example}

We conclude by generalizing Example~\ref{ex: hyperplane} to higher dimensional hyperplanes.

\begin{proposition}
\label{prop:g-equivalence}
    The solutions to the homaloidal PDE for ML degree one hyperplanes in $\Sym(\CC^m)$ are the rational functions
    \begin{align*}
    \Phi(S) = \frac{    
    \det\left((gSg^\top)_{[n]\setminus \{1,2\}} \right)}{  \det\left((gSg^\top)_{[n]\setminus \{1\}} \right)\det\left((gSg^\top)_{[n]\setminus \{2\}} \right)},
    \end{align*}
where $g \in \SL(\mathcal L)$.
\end{proposition}

\begin{proof}
The action $g \cdot K = gKg^{\top}$ gives a containment $\mathrm{SL}(\CC^m) \subset \mathrm{SL}(\Sym(\CC^m))$. The hyperplane $X= V(\mathrm{tr}(AK))$ is ML degree one if and only if $A$ is rank two, 
by \cite[Proposition 4.3]{AGKMS21}. All such hyperplanes are in the same orbit under the action of $\mathrm{SL}(\Sym(\CC^m))$. Hence there exists $g$ such that \[ gAg^{\top} = \left( \begin{array}{@{}c|c@{}}
\begin{matrix} 0 & 1 \\ 1 & 0 \end{matrix} & \begin{matrix} \bf{0} \\ \bf{0} \end{matrix} \\
\hline
\begin{matrix} \bf{0} & \bf{0} \end{matrix} & \begin{matrix} \bf{0} \end{matrix} \\
\end{array} \right),\] 
where $\bf{0}$ is a matrix of zeros of appropriate size.
Thus $X$ is $\mathrm{SL}(\CC^m)-$equivalent to the graphical model $k_{12}=0$. 

We next show a more general statement, namely that if $F$ is invariant under the action of $g\in \SL(\mathcal L)$, then $\myMLD_F(X) = 1$ implies $\myMLD_F(g\cdot X) = 1$. Moreover,
$
\Phi_{g\cdot X, \vectorspace, F} \, = \,  \Phi_{X, \vectorspace, F} \circ g^{\top},
$
where $g^\top(\bu)\coloneqq \bu\circ g$. 
    The map 
    $$\_ \circ g^\top: (\vectorspace^\vee)^\vee \to (\vectorspace^\vee)^\vee
    $$ is the double-dual $(g^{\top})^{\top}$, which is identified with $g$. We use the chain rule to compute
\[
    \nabla_{\bu} (\log \Phi_{X, \vectorspace, F} \circ g^{\top}) = 
    (\nabla_{g^{\top}(\bu)} \log \Phi_{X, \vectorspace, F}) \circ g^{\top} = 
    g\cdot (\nabla_{g^{\top}(\bu)} \log \Phi_{X, \vectorspace, F}). 
\]
The model associated to $\Phi_{X, \vectorspace, F} \circ g^{\top}$ is $g \cdot X$, by Theorem \ref{thm:pde}. The function $\Phi_{X, \vectorspace, F} \circ g^{\top}$ is a solution to homaloidal PDE, since $F$ is invariant under the action of $g$, and therefore 
\begin{align*}
     F(-\nabla_{\bu} (\log \Phi_{X, \vectorspace, F} \circ g^{\top} ))  
    & = F(g \cdot (-\nabla_{g^{\top}(\bu)} \log \Phi_{X, \vectorspace, F}))  \\
    &= F (-\nabla_{g^{\top}(\bu)} \log \Phi_{X, \vectorspace, F})  = \Phi_{X, \vectorspace, F} \circ g^{\top}.
\end{align*}
It remains to apply our group action to the expression for $\Phi$ in Proposition \ref{prop: chordal-mle}.
\end{proof}

\bigskip

\subsection*{Acknowledgments}
KK was supported by the Wallenberg AI, Autonomous Systems and Software Program (WASP) funded by the Knut and Alice Wallenberg Foundation.
LG was supported by the VR grant [NT:2018-03688]. OM was supported by Brummer \& Partners MathDataLab, Göran Gustafssons Stiftelse UU/KTH, and the European Research Executive Agency (101061315MIAS-HORIZON-MSCA-2021-PF-01).
AS was supported by the Society of Fellows at Harvard University.

\ifarxiv
\bibliographystyle{alpha}
\else
\bibliographystyle{journalbibstyle}
\fi
\bibliography{literature}

\begin{landscape}
\begin{table}[h]
\label{table:pde-examples}
\begin{center}
\begin{tabular}{c*{4}{>{$}Sc <{$}}cc}
\toprule
Ex.&\vectorspace & F & X & \makecell{\Phi_{X,\vectorspace,F}} &
Description & General result
\\
\midrule
\ref{ex: 1-2-3} & \Sym (\CC^3) & \det & V(k_{13}) & \dfrac{\det(S_2)}{\det(S_{12})\det(S_{23})} & \makecell{Undirected \\ graphical model} & Prop.~\ref{prop: chordal-mle}\\
\ref{ex: unshieldedcollider} & \Sym (\CC^3) & \det & V(k_{13}k_{23} - k_{12}k_{33}) & \dfrac{s_{11}s_{22}-s_{12}^2}{s_{11}s_{22}\det(S)} & \makecell{Directed \\ graphical model} & Prop.~\ref{prop:directed_phi}\\
\ref{ex:specialquadrics}& \mathbb{P}^2  & \ell & V(Q)  & \dfrac{\bu(p)}{Q^\vee(\bu)} & \makecell{Quadric curve, \\ $F$ linear} & Thm.~\ref{thm:linear_case} \\
\ref{ex: smoothquadric} & \mathbb{P}^n & Q & \mathbb{P}^n & \dfrac{4}{Q^\vee(\bu)} & \makecell{Linear space, \\ $F$ quadratic} & Thm.~\ref{huh-result}  \\
\ref{ex:conictangent} & \mathbb P^n & Q\ell & \mathbb P^n &  16\dfrac{\bu(p)}{Q^\vee(\bu)^2} &
\makecell{Linear space, \\ $F$ special}  & \makecell{Thm.~\ref{plane-curves} \\ ($n=2$)} \\
\ref{ex:squareandcorner} & \Sym (\CC^3) & \det & V(k_{32},k_{31},k_{21}^2-k_{11}k_{22}+k_{11}k_{33}) & \makecell{\dfrac{-4s_{22}}{(s_{22} + s_{33})^2(s_{12}^2 - s_{11}s_{22})}} & \makecell{ Reciprocal \\ linear space} & – \\
\ref{ex: hyperplane} & \Sym (\CC^3) & \det & V(k_{11} - k_{22}) & \dfrac{4s_{33}}{\det (g\cdot S)_{23}\det (g\cdot S)_{13}} & \makecell{  Hyperplane} & Prop.~\ref{prop:g-equivalence} \\
\bottomrule
\end{tabular}
\end{center}
\caption{Some varieties of ML degree one.} 
\end{table}
\end{landscape}

\end{document}